\documentclass[reqno,10pt]{amsart}

\usepackage{pdfsync}
\usepackage{ifthen}
\usepackage{amsmath,paralist}
\usepackage{amsthm}
\usepackage{graphicx,eucal}
\usepackage{graphpap,latexsym,epsf,bm}
\usepackage{color}
\usepackage{hyperref}
\usepackage{tikz}
\usepackage{graphicx}
\usepackage{psfrag, paralist}
\usepackage{amssymb,mathrsfs,enumerate,paralist}

\numberwithin{equation}{section}

\newtheorem{theorem}{Theorem}[section]
\newtheorem{corollary}[theorem]{Corollary}
\newtheorem{lemma}[theorem]{Lemma}
\newtheorem{proposition}[theorem]{Proposition}
\newtheorem{definition}[theorem]{Definition}
\newtheorem{example}[theorem]{Example}

\newtheorem{hypothesis}[theorem]{Hypothesis}

\theoremstyle{remark}
\newtheorem{remark}[theorem]{Remark}

\newcommand{\N}{\mathbb{N}} 
\newcommand{\Z}{\mathbb{Z}}
\newcommand{\R}{\mathbb{R}} 

\newcommand{\teta}{\vartheta}
\newcommand{\dd}{\,\mathrm{d}}

\newcommand{\e}{\varepsilon}
\newcommand{\eps}{\varepsilon}

\newcommand{\weaksto}{{\rightharpoonup^*}\,}
\newcommand{\weakto}{\rightharpoonup}
\newcommand{\Restr}[1]{\lower2pt\hbox{$|_{#1}$}}
\newcommand{\argmin}{\mathop{\mathrm{Argmin}}}

\newcommand{\ene}[2]{\mathcal{E}(#1, #2)}

\newcommand{\Varname}[1]{{\mathrm{Var}}_{#1}}

\newcommand{\Var}[4]{\mathrm{Var}_{#1}(#2; [#3,#4])}

\newcommand{\BV}{\ensuremath{\mathrm{BV}}}
\newcommand{\AC}{\mathrm{AC}}

\newcommand{\vvmV}[3]{
  \fre_{#1}(#2)
  \ifthenelse{\equal{#3}{}} {} {\|#3\|} }

\newcommand{\FNormname}[1]{\mathscr{P}}

\newcommand{\Leb}[1]{\mathscr{L}^{#1}}

\newcommand{\RIS}{\ensuremath{( \psiri,\bptname,\mathcal{E})}}

\newcommand{\Cost}[5]{\Delta_{#1,#2}(#3; #4, #5)}

\newcommand{\essJ}{\mathop{\text{\upshape ess-J}}\nolimits}

\newcommand{\foraa}{\text{for a.a.\,}}

\newcommand{\down}{\downarrow}
\newcommand{\up}{\uparrow}

\newcommand{\rmC}{\mathrm{C}}
\newcommand{\rmD}{\mathrm{D}}

\newcommand{\Vtaue}[1]{\mathrm{V}_{\kern-1pt\tau,\eps}^{#1}}
\newcommand{\Xitaue}[1]{\Xi_{\kern-1pt\tau,\eps}^{#1}}

\newcommand{\piecewiseConstant}[2]{\overline{\mathrm{#1}}_{\kern-1pt#2}}

\newcommand{\underpiecewiseConstant}[2]{\underline{\mathrm{#1}}_{\kern-1pt#2}}

\newcommand{\piecewiseLinear}[2]{\mathrm{#1}_{\kern-1pt#2}}

\newcommand{\pwM}[2]{\widetilde{\mathrm{#1}}_{\kern-1pt#2}}
 \def\trait #1 #2 #3 {\vrule width #1pt height #2pt depth #3pt}

\newcommand{\fre}{\mathfrak e}

\newcommand{\Diss}[2]{
  \ifthenelse      {   \equal{#1}{V}  }         {\Psiv}
  {
    \ifthenelse      {   \equal{#1}{B}  }         {\Psiz}
    {
      \ifthenelse      {   \equal{#1}{\Bo}  }     {\Psiz}
      {\Psi_{\kern-1pt#1}}
      }
    }
  ^{#2}
}

\newcommand{\Dnorm}[1]{\Psiz_{\kern-2pt \scriptscriptstyle\land}(#1)}
\newcommand{\Dnormname}{\Psiz_{\kern-2pt \scriptscriptstyle\land}}

\newcommand{\cE}{\mathcal{E}}
\newcommand{\calE}{\cE}

\def\nchi{{\raise.3ex\hbox{$\chi$}}}
\newcommand{\Bo}{{}}

\newcommand{\rmV}{\mathrm V}
\newcommand{\scalarV}[3]{\rmV_{\!#2}
  \ifthenelse{\equal{#3}{}}{}{(#3)}}
\newcommand{\scalarmu}[2]{\mu}
\newcommand{\scalarmuco}[2]{\mu_{\rm d}}
\newcommand{\scalarmuj}[2]{\mu_{\rm J}}
\newcommand{\scalarmul}[2]{\mu_{\mathscr L}}
\newcommand{\scalarmuc}[2]{\mu_{\rm C}}

\newcommand{\Psiv}{\Phi}
\newcommand{\Psiz}{\Psi}

\def\ss{\mathsf s}
\def\st{\mathsf t}
\def\su{\mathsf u}
\def\e{\varepsilon}

\def\J{\mathcal{J}}

\def\f{_{\infty}}
\def\E{\mathcal{E}}

\def\gl{\Gamma\text{-}}

\newcommand{\QED}{\mbox{}\hfill\rule{5pt}{5pt}\medskip\par}

\newcommand{\bpt}[3]{\mathsf{p} (#1,#2,#3)}
\newcommand{\bptname}{\mathsf{p}}

\newcommand{\bptnew}[3]{\mathsf{p} (#1,#2,#3)}

\newcommand{\Jfu}[2]{\mathscr{J}_{#1}(#2)}
\newcommand{\Jfuname}[1]{\mathscr{J}_{#1}}
\newcommand{\tilJfunname}[1]{\tilde{\mathscr{J}}_{#1}}

\newcommand{\bip}[4]{\mathsf{b}_{#1}(#2,#3,#4)}
\newcommand{\bipd}[5]{\mathsf{b}_{#1}^{#2}(#3,#4,#5)}
\newcommand{\bipname}[1]{\mathsf{b}_{#1}}
\newcommand{\bipnamed}[2]{\mathsf{b}_{#1}^{#2}}

\newcommand{\Ctc}[1]{\Lambda_{#1}}
\newcommand{\tCtc}[1]{\widetilde{\Lambda}_{#1}}

\newcommand{\co}{\mathrm{co}}
\newcommand{\ju}[1]{\mathrm{J}_{#1}}
\newcommand{\Jump}[4]{\mathrm{Jmp}_{#1}(#2; [#3,#4])}
\newcommand{\Ca}{{\rm C}}
\newcommand{\Le}{\mathscr{L}}

\newcommand{\rmH}{\mathrm{H}}
\newcommand{\bx}{x}
\newcommand{\be}{\mathbf e}
\newcommand{\bw}{w}
\newcommand{\bv}{v}

\newcommand{\zzeta}{\zeta}
\newcommand{\bxi}{\xi}
\newcommand{\btheta}{\theta}

\newcommand{\psiri}{\Psi_0}

\definecolor{ddmagenta}{rgb}{0.7,0,0.9}
\definecolor{ddcyan}{rgb}{0,0.2,1.0}

\newenvironment{rnew}{\color{red}}{\color{black}}
\newcommand{\berin}{\begin{rnew}}
\newcommand{\erin}{\end{rnew}}

\newenvironment{ric}{\color{ddcyan}}{\color{black}}
\newcommand{\beric}{\begin{ric}}
\newcommand{\eric}{\end{ric}}

\newcommand{\RRR}{\color{black}}
\newcommand{\EEE}{\color{black}}

\newcommand{\HIDE}{\color{ddcyan}}

\begin{document}

\title[Generation  of Balanced
Viscosity solutions to rate-independent systems]
{Generation via variational convergence of Balanced
Viscosity solutions to rate-independent systems}

\author{Giovanni A. Bonaschi}
\address{Giovanni A. Bonaschi \\ Technische Universiteit Eindhoven, 5600 MB, Eindhoven, The Netherlands, and
\\
 AVIVA, Milano, Italy.}
\email{giovanni.bonaschi\,@\,gmail.com}

\author{Riccarda Rossi}
\address{Riccarda Rossi \\ DIMI, Universit\`a di
  Brescia, via Branze 38, I--25133 Brescia, Italy.}
\email{riccarda.rossi\,@\,unibs.it}


\date{October 09, 2017}

\thanks{G.A.B.\ kindly acknowledges support from the Nederlandse Oxrganisatie voor Wetenschappelijk Onderzoek (NWO),  VICI Grant 639.033.008}
\thanks{R.R.\
has  been partially supported  by GNAMPA (Gruppo Nazionale per l'Analisi Matematica, la Probabilit\`a e le loro Applicazioni)  of  INdAM (Istituto Nazionale di Alta Matematica).}

\begin{abstract}
In this paper we investigate the origin of the Balanced Viscosity solution concept for  rate-independent evolution in the setting of a finite-dimensional space. Namely, given a family of dissipation potentials $(\Psi_n)_n$ with superlinear growth at infinity and a smooth energy functional $\calE$, we enucleate sufficient conditions on them ensuring that the associated gradient
systems $(\Psi_n,\calE)$  \emph{Evolutionary Gamma-converge}, cf.\
\cite{Mie-evolGamma},  to a limiting rate-independent system, understood in the sense of  \emph{Balanced Viscosity} solutions. In particular, our analysis encompasses both  the \emph{vanishing-viscosity} approximation of rate-independent systems  from  \cite{MRS10,mielke-rossi-savare2013},
 and their \emph{stochastic} derivation  developed in 
\cite{BonaschiPeletier14}. 
\end{abstract}
\maketitle


\section{\bf Introduction}
Over the last years,
rate-independent systems 
have been the object of intensive mathematical investigations. This is undoubtedly due to their vast range of applicability.
Indeed, this kind of processes seems to be ubiquitous in continuum mechanics, ranging from shape memory alloys to crack propagation, from elastoplasticity to 
damage and  delamination. 
They also occur  in fields such as ferromagnetism and   ferroelectricity.   We refer to \cite{Miesurvey, MieRouBOOK} for   a thorough survey  of all these problems. 
\par
Besides its applicative  relevance, though,  rate-independent evolution has an own, intrinsic, mathematical interest. This is  apparent already  in the context of a \emph{finite-dimensional} rate-independent system, 
driven
 by a \emph{dissipation potential} $\psiri : \R^d \to [0,+\infty)$ (non-degenerate), convex, and   positively homogeneous of degree $1$, and 
 an \emph{energy functional} $\calE : [0,T]\times \R^d \to \R$; in particular,  throughout the paper, we will consider a smooth energy
$\calE$
 such that the power \RRR function \EEE $\partial_t \calE $ is controlled by  $\calE$ itself,
  namely
\begin{equation}
\label{smooth-energy}
\tag{$\textsc{E}$}
\calE \in \rmC^1 ([0,T]\times\R^d)\quad  \text{ and  } \quad \exists\, C_1,\, C_2>0 \ \
\forall\, (t,u) \in [0,T]\times \R^d \, : \qquad |\partial_t \ene tu | \leq C_1 \ene tu + C_2\,.
\end{equation}
 The pair $(\psiri,\calE) $ give rise to the simplest example of rate-independent evolution, namely  the
gradient system
\begin{equation}
\label{simplest-ris-intro}
\partial\psiri(u'(t)) + \rmD \ene t{u(t)} \ni 0 \qquad \foraa\, t \in (0,T),
\end{equation}
where $\partial\psiri: \R^d \rightrightarrows \R^d $ is the subdifferential of $\psiri$ in the sense of convex analysis, whereas $\rmD \calE$ is  the differential of 
the map $u \mapsto \ene tu$. 
Due to the $0$-homogeneity of $\partial\psiri$, \eqref{simplest-ris-intro} is invariant for time rescalings, i.e.\
it is 
 rate-independent. 
Now, it is well known that, even in the case of a \emph{smooth} energy $\calE$, if $u\mapsto \ene tu$ fails to be strictly convex, then absolutely continuous
solutions to \eqref{simplest-ris-intro} need not exist. 
In the last two decades, this has motivated  the development of various weak solvability concepts for \eqref{simplest-ris-intro} and, in general, for rate-independent
systems in infinite-dimensional Banach spaces, or even topological spaces.  The analysis of these solution notions has posed several challenges.
\paragraph{\bf \emph{Energetic} and \emph{Balanced Viscosity} solutions to rate-independent systems.}
While referring to  \cite{Miel08?DEMF} and  \cite{MieRouBOOK} for a  survey of all weak  notions of   rate-independent evolution, we may recall here the concept  of \emph{Energetic solution}, 
first proposed in  \cite{Mie-Theil99}  (cf.\ also \cite{DMToa02}  for the concept of \emph{quasistatic evolution} in fracture),  and fully analyzed in \cite{Mielke-Theil04}. 
It consists  of 
the \emph{global stability}  condition
\begin{equation}
  \label{eq:1-intro}
  \forall\, z\in \R^d:\qquad \ene t{u(t)}\le \ene
  tz+\psiri(z-u(t)) \qquad \text{for every } t \in [0,T], 
  \tag{$\mathrm{S}$}
\end{equation}
and of the $(\psiri,\calE)$-\emph{energy balance}
\begin{equation}
  \label{eq:2-intro}
  \ene{t}{u(t)}+\Var{\psiri}u{0}{t}=\ene{0}{u(0)}+
  \int_{0}^{t} \partial_t\ene s{u(s)}\,\mathrm{d}s \qquad \text{for every } t \in [0,T],
  \tag{$\mathrm{E}_{\psiri,\calE}$}
\end{equation}
involving the dissipated energy $\Var{\psiri}u{0}{t}$ (where $\Varname{\psiri}$ denotes the notion 
of total variation induced by $\psiri$), the stored energy $\ene t{u(t)}$ at the process time $t$, the initial energy $\ene 0{u(0)}$,  
 and the work of the external forces.
 Since the energetic formulation
\eqref{eq:1-intro}--\eqref{eq:2-intro} only
features the (assumedly smooth) power of the external forces
$\partial_t \mathcal{E}$, and no other derivatives,
it is particularly suited to solutions with  discontinuities in time. It is also 
considerably flexible and can be indeed given for rate-independent processes 
in general topological spaces, cf.\ \cite{MainikMielke05}.  That is why, it has been exploited in a great variety  of applicative contexts, cf.\  \cite{Miesurvey,MieRouBOOK}. 

Nonetheless,  over the years it has become apparent that, in the very case of a \emph{nonconvex} dependence $u\mapsto \ene tu$,  the \emph{global} stability
\eqref{eq:1-intro} fails  provide a truthful description of  the system behaviour at jumps, leading to solutions jumping `too early' 
and `too long' (i.e.\ into very far-apart energetic configurations), as shown
for instance
 by the examples   \cite[Ex.\,6.1]{Miel03EFME},
and
\cite[Ex.\,1]{MRS09}, and by  the characterization of energetic solutions to
(one-dimensional) rate-independent systems in  \cite{Rossi-Savare13}. 
\par
This circumstance 
has led to  the introduction of \emph{alternative} weak solvability concepts for \eqref{simplest-ris-intro}
 and its generalizations. The focus of this paper is on the notion of \emph{Balanced Viscosity} solution, first introduced in \cite{MRS10}   for a  \emph{finite-dimensional} rate-independent system and later extended to the infinite-dimensional case in \cite{mielke-rossi-savare2013}.
 The origin of this concept in fact goes back to the seminal paper \cite{ef-mie06}, which first set forth \emph{vanishing viscosity} as a selection criterion for 
 mechanically feasible weak solution notions to rate-independent systems.  The vanishing-viscosity approach has in fact proved to be a robust method in manifold applications, e.g.\ ranging from plasticity
 \cite{DalDesSol11,BabFraMor12,FrSt2013}, to  fracture  \cite{KnMiZa07?ILMC,LazzaroniToader}, and to
 damage  \cite{krz, CL16} models.  We also refer to \cite{Negri14} for an alternative derivation of Balanced Viscosity solutions via time discretization. 
 \par
 Let us briefly illustrate the vanishing-viscosity approach:
 We ``augment by viscosity''  the dissipation potential $\psiri$ and thus introduce 
\begin{equation}
\label{eps-pot-intro}
\Psi_\eps(v): = \psiri(v) + \frac{\eps}{2} \|v\|^2,
\end{equation} 
 with $\| \cdot \|$ a second norm on $\R^d$, possibly coinciding with $\psiri$, and the corresponding gradient system $(\Psi_\eps,\calE)$, namely the doubly nonlinear equation
 \begin{equation}
 \label{grsys-eps-intro}
 \partial \Psi_\eps(u'(t)) + \rmD \ene t{u(t)} \ni 0 \qquad \foraa\, t \in (0,T). 
 \end{equation}
 Since $\Psi_\eps$ has superlinear growth at infinity, 
  \eqref{grsys-eps-intro} does admit absolutely continuous solutions. It is to be expected that, as the viscosity parameter $\eps$ vanishes, 
  solutions $(u_\eps)_\eps$ to 
\eqref{grsys-eps-intro} will converge to a suitable weak solution to the rate-independent system \eqref{simplest-ris-intro}. In \cite{MRS10} 
it was indeed shown that any limit curve $u\in \BV ([0,T];\R^d)$ of the functions $(u_\eps)_\eps$ complies with 
the stability condition 
  \begin{equation}
    \label{eq:65bis-intro}
    \tag{$\mathrm{S}_\mathrm{loc}$}
    -\rmD\ene t{u(t)}\in K^*: = \partial \psiri(0) \quad \foraa\, t \in (0,T), 
  \end{equation}
  and with the energy balance 
  \begin{equation}
    \label{eq:84-intro}
    \Var{\psiri,\mathfrak{p},\calE}u{0}{t}+\ene{t}{u(t)}=\ene{0}{u(0)}+
    \int_{0}^{t} \partial_t\ene s{u(s)}\,\dd s \quad \text{ for all } t\in [0,T]\,.
    \tag{$\mathrm{E}_{\psiri, \mathfrak{p},\calE}$}
  \end{equation}
  Although  \eqref{eq:65bis-intro} \& \eqref{eq:84-intro} look similar to  \eqref{eq:1-intro} \&  \eqref{eq:2-intro}, they are in fact significantly different. 
  First of all,  \eqref{eq:65bis-intro}
 is in fact a \emph{local} version of the global stability \eqref{eq:1-intro}. Secondly, \eqref{eq:84-intro} shares the same structure with the energy balance  \eqref{eq:2-intro}, but it features a notion of total variation involving, in addition to  the dissipation potential 
  $\psiri$, the \emph{vanishing-viscosity contact potential}
  \begin{equation}
  \label{vvcp-intro}
  \mathfrak{p}(v,\xi): = \inf_{\eps>0} \left( \Psi_\eps(v) + \Psi_\eps^*(\xi)\right) = \psiri(v) + \|v\| \min_{\zeta\in K^*}\|\xi-\zeta\|_{*},
  \end{equation}
 (with $K^*$ from \eqref{eq:65bis-intro}   and $\| \cdot\|_*$  the dual norm of $\| \cdot\|$).  
  While referring to Section \ref{s:3} for the precise definition of 
  $ \Varname{\psiri,\bptname,\calE}$,
cf.\ \eqref{pseudo-Var},
   we may mention here that  $\mathfrak{p}$ indeed encodes how viscosity, neglected in the vanishing-viscosity limit,
  pops back into the description of the solution behaviour at jumps, whereas in the continuous (`sliding') regime, the system is only governed by the dissipation $\psiri$. 
  \par
    A characterization of Balanced Viscosity solutions, again for one-dimensional systems, 
  has been provided in \cite{Rossi-Savare13}, showing that they model  jumps more accurately than energetic solutions. On the other hand, as evident from \eqref{vvcp-intro}, this notion seems to be strongly reminiscent of the vanishing-viscosity approximation \eqref{grsys-eps-intro}. 
  \par
  It is thus natural to wonder if there are ways, \emph{alternative} to the vanishing-viscosity 
  \cite{MRS10, mielke-rossi-savare2013} and to the time-discretization \cite{Negri14} approaches,  
  to generate 
  Balanced Viscosity solutions. 
  \paragraph{\bf The stochastic origin of Balanced Viscosity solutions.}
  Recently, this question has been answered affirmatively in \cite{BonaschiPeletier14}, investigating the role of stochasticity in the origin of rate-independence,
  in the \emph{one-dimensional} setting (we refer to \cite{MPR14} for analogous results on the origins of generalized gradient structures).    
   More
  specifically,  \cite{BonaschiPeletier14} has  focused on a continuous-time Markov jump process
  $t\mapsto X_t^h$ on a one-dimensional lattice, with lattice spacing $\frac1h$, $h\in \N$.  While referring to  Section \ref{s:2} for more details, we may mention here
  that this process
  models the  evolution of a Brownian particle in a wiggly energy landscape, involving the energy $\calE$,   in the following way.
  If the particle is at the position $x$ at time $t$, then it jumps in continuous time to  its neighbours $x\pm \frac1h$ with rates $h r^{\pm}(x)$, where
  $r^\pm(x)   = \alpha \exp(\mp \beta \rmD \ene tx)$. Here, $\alpha$ and $\beta$ are positive parameters, 
  the former characterizing the rate of jumps, and thus the global time scale of the process, and the latter related to noise. 
  \par
 First of all,  in \cite{BonaschiPeletier14} it was shown that the deterministic limit, in a `large-deviations' sense,   as $h\to\infty$ and for $\alpha$ and $\beta$ fixed, 
  of this stochastic process solves the gradient system
  \[
  u'(t) = 2\alpha \sinh(-\beta \rmD \ene t{u(t)}) \qquad \foraa\, t \in (0,T).
  \]
  Observe that the latter is a reformulation of the 
 doubly nonlinear evolution equation 
  \begin{equation}
  \label{gr-alphabeta-intro}
  \partial \Psi_{\alpha,\beta}(u'(t)) + \rmD \ene t{u(t)} \ni0 \qquad  \foraa\, t \in (0,T),
  \end{equation}
   where the dissipation potential $\Psi_{\alpha,\beta}$ is such that its Fenchel-Moreau convex conjugate 
  fulfills
  $\partial\Psi_{\alpha,\beta}^* (\xi) = \{ \rmD \Psi_{\alpha,\beta}^* (\xi) \} = \{2\alpha \sinh(\beta\xi)\}$.
  More precisely, in \cite{BonaschiPeletier14} it was proved that the process $X^h$ satisfies a \emph{large deviations principle}, with rate function 
  given by the functional of trajectories $\tilJfunname{\Psi_{\alpha,\beta}, \calE}: \BV ([0,T];\R^d) \to  [0,+\infty]$ 
 defined by
  \[
 \tilJfunname{\Psi_{\alpha,\beta}, \calE} (u) : = 
 \beta \left( \int_0^T \left( \Psi_{\alpha,\beta} (u'(t)) + \Psi_{\alpha,\beta}^*(-\rmD \ene t{u(t)})  \right) \dd t + \ene T{u(T)} - \ene 0{u(0)} - \int_0^T \partial_t \ene t{u(t)} \dd t  \right) 
 \]  
  if $u\in \AC([0,T];\R^d)$, and $+\infty$ else.  It is easy to check that the (null-)minimizers of $ \tilJfunname{\Psi_{\alpha,\beta}, \calE} $ are solutions to 
  the gradient system
  \eqref{gr-alphabeta-intro}.
  \par
  Next, the \emph{variational} limits of the functionals $\tilJfunname{\Psi_{\alpha,\beta}, \calE} $ have been  addressed under different scalings of the parameters $\alpha $ and $\beta$, leading to  gradient flow or  rate-independent evolution. 
  To illustrate the result in the  latter case, here and throughout the paper we will confine the discussion
  to the following choice of parameters: $\alpha=\alpha_n := \frac{e^{-nA}}2$ and $\beta: = \beta_n =n$, with $n\in\N$. 
  Therefore, 
   the associated dissipation potentials
  are given by
  \begin{equation}
  \label{psi-ld-intro}
 \Psi_n(v): = \Psi_{\alpha_n,\beta_n}(v) =   \frac{v}{n} \log \left( \frac{ v + \sqrt{v^2 + e^{-2n A}}}{e^{-n A}} \right) - \frac{1}{n}\sqrt{v^2 + e^{-2n A}} + \frac{e^{-n A}}{n}.
  \end{equation}
   In \cite[Thm.\ 4.2]{BonaschiPeletier14} it was then proved that that the 
  functionals $\Jfuname{\Psi_n,\calE}: = \frac1n  \tilJfunname{\Psi_n, \calE}$ converge in the sense of \textsc{Mosco}, 
  with respect to the \emph{weak-strict} topology in 
  $ \BV ([0,T];\R^d) $, to 
  the functional  $  \Jfuname{\psiri,\mathfrak{p},\calE} : \BV([0,T];\R^d) \to [0,+\infty]$ defined by
  \begin{equation}
  \label{Jfu-BV-intro}
  \begin{aligned}
  \Jfu{\psiri,\mathfrak{p},\calE}u:   =
 & \Var{\psiri,\mathfrak{p},\calE}{u}{0}T +  \int_0^T  I_{K^*}(-\rmD \ene t{u(t)}) \dd t   
 \\ & \quad 
 +\ene{T}{u(T)} - \ene{0}{u(0)} -     \int_{0}^{T} \partial_t\ene s{u(s)}\,\dd s,
 \end{aligned}
\end{equation}
with $\psiri(v) = A|v|$, $\mathfrak{p}$ given by \eqref{vvcp-intro}
 and the associated total variation functional  $ \Varname{\psiri,\bptname,\calE}$ defined in
   \eqref{pseudo-Var} ahead, 
  and with 
$I_{K^*} $ denoting the indicator function of the set $K^*=[-A,A]$. Recall that \textsc{Mosco}-convergence (cf.\ e.g.\ \cite{Attouch}) 
  with respect to the weak-strict topology in 
  $ \BV ([0,T];\R^d)$
means that  
\begin{equation}
\label{Mosco}
\begin{aligned}
 & \text{\emph{(i)}} && u_n \to u \text{ weakly in $\BV([0,T];\R^d)$ }  \  \Rightarrow \  \liminf_{n\to\infty} \Jfuname{\Psi_n,\calE}(u_n) \geq  \Jfu{\psiri,\mathfrak{p},\calE}u,
\\
 & \text{\emph{(ii)}} && \forall\, u \in \BV([0,T];\R^d) \ \exists\, (u_n)_n \subset  \BV([0,T];\R^d) \text{ s.t. }   
 \begin{cases}
 &
 u_n\to u \text{ strictly in $\BV([0,T];\R^d)$ }, 
 \\
 &
 \limsup_{n\to\infty} \Jfuname{\Psi_n,\calE}(u_n) \leq  \Jfu{\psiri,\mathfrak{p},\calE}u\,.
\end{cases}
\end{aligned}
\end{equation} 
  Since the (null-)minimizers of $ \Jfuname{\psiri,\mathfrak{p},\calE}$ are Balanced Viscosity solutions of the rate-independent system driven by $\psiri$ and $\calE$ (cf.\ 
  Proposition \ref{prop:null-minim} ahead),  
  \cite[Thm.\ 4.2]{BonaschiPeletier14}  ultimately establishes a connection between the jump process $X^h$ and the latter  rate-independent system, understood in a \emph{Balanced Viscosity} sense. Furthermore, observe that the functionals $\Psi_n$ from \eqref{psi-ld-intro} are not of the form \eqref{eps-pot-intro}. Therefore,
  this result provides a way, alternative to vanishing viscosity, to generate  Balanced Viscosity solutions.
  \paragraph{\bf Our results.}
  The aim of this paper is twofold.   
  \par
  First of all, we intend to extend the `stochastic generation' of Balanced Viscosity solutions investigated in \cite{BonaschiPeletier14}, to the \emph{multi-dimensional}
  rate-independent system \eqref{simplest-ris-intro}, where now
  \[
  \psiri(v): = A\|v\|_1   \text{ for all } v \in \R^d,   \qquad \text{with } \|v\|_1: = \sum_{i=1}^d |v_i|\,.
  \]
Even conjecturing that the viscosity contact potential defining  the limiting Balanced Viscosity solution notion 
 is of the form 
\eqref{vvcp-intro}, 
 in the multi-dimensional case it is no longer obvious which choice of the viscous norm $\| \cdot \|$ should enter into 
\eqref{vvcp-intro}. Indeed, 
with \textbf{our main results,} \underline{\bf Theorem \ref{thm:liminf}} ($\liminf$-estimate) and \underline{\bf Thm.\ \ref{thm:limsupgenerale}}  ($\limsup$-estimate), 
we will show that  the  multi-dimensional analogues of   the functionals $(\Jfuname{\Psi_n,\calE})_n $ \textsc{Mosco}-converge, with respect to the weak-strict topology of $\mathrm{BV}([0,T];\R^d)$,  to the functional $ \Jfuname{\psiri,\mathfrak{p},\calE}$ featuring the contact potential
\begin{equation}
\label{p-limit-intro}
\mathfrak{p}(v,\xi) := \|v\|_1 (A \vee \| \xi\|_\infty) \qquad \text{with }  \| \xi \|_\infty :=
\max_{i=1, \ldots, d} |\xi_i |.
\end{equation}
 It can be checked that $\mathfrak{p}$ is indeed of the form \eqref{vvcp-intro},  with  the `viscous' norm 
$\| \cdot \|$ in fact coinciding with that associated with $\psiri$, i.e.\
 $\|v\|= \psiri(v) =A\|v\|_1 $. Namely,  the notion of 
 Balanced Viscosity solution arising from the stochastic approximation coincides with the one obtained by 
 vanishing \emph{$\psiri$-viscosity},  cf.\  Example \ref{ex:van-visc} ahead. 
 \par
 Secondly, we shall investigate on a more general and deeper level the origin of rate-independent evolution in a Balanced Viscosity sense. More precisely, 
 \begin{compactitem}
 \item
  we
 will introduce an `extended' notion of Balanced Viscosity solution, induced by  a general  \emph{viscosity contact potential} $\bptname : [0,+\infty) \times \R^d \times \R^d \to [0+\infty]$, 
$ \bptname = 
 \bpt \tau v \xi$, cf.\ Def.\ \ref{def:visc-cont-pot} ahead,   such that 
 the contact potentials $\mathfrak{p}(v,\xi) $ from \eqref{vvcp-intro} are obtained for $\tau=0$, i.e.\   $\bpt 0 v \xi = \mathfrak{p}(v,\xi)$ (i.e., $\mathfrak{p}$ is  augmented of the time variable);
 \item
  we  will enucleate a series  of  conditions under which a sequence $(\Psi_n)_n$ of \emph{general} dissipation potentials with superlinear growth at infinity, 
 not necessarily of the form \eqref{eps-pot-intro} (vanishing-viscosity) or \eqref{psi-ld-intro}  (stochastic approximation), 
  give rise to a viscosity contact potential.  Such conditions will amount to requiring that 
 the 
 bipotentials 
 $b_{\Psi_n} : [0,+\infty) \times \R^d \times \R^d   \to [0+\infty]$,   associated with the functionals $\Psi_n$, and defined for $\tau>0$ by $b_{\Psi_n}(\tau,v,\xi): = \tau \Psi_n(\frac v\tau) + \tau \Psi_n^*(\xi)$   (cf.\  \eqref{def-bptreg}
 ahead),  converge in a suitable \emph{variational sense} to $\bptname$. 
 \item
 It will turn out (cf.\ Theorem \ref{thm:liminf}), that under this condition, joint with a suitable uniform coercivity requirement for the functionals $(\Psi_n)_n$, 
 the $\Gamma$-$\liminf$ estimate  in \eqref{Mosco} 
  for the associated trajectory functionals $(\Jfuname{\Psi_n,\calE})_n$ 
 holds.
 \item
  As we will see,  this  implies that 
 limit curves $u\in \BV ([0,T];\R^d)$ of sequences of solutions $(u_n)_n$ to the gradient systems $(\Psi_n,\calE)$ are Balanced Viscosity solutions to the rate-independent system $(\psiri,\bptname,\calE)$,  i.e.\  
  systems $(\Psi_n,\calE)_n$ Evolutionary $\Gamma$-converge, in the sense of \cite{Mie-evolGamma}, to $(\psiri,\bptname,\calE)$.  
 \end{compactitem}
 \par
Let us clarify that,
 the  $\liminf$-estimate in Theorem \ref{thm:liminf} will be valid in the  general setting specified  in the above lines,
and could be in fact trivially proved  for systems set in \emph{any} abstract  finite-dimensional Banach space $\mathbf{X}$. Instead, 
    in Theorem  \ref{thm:limsupgenerale}  we will  be  able to prove the $\Gamma$-$\limsup$ inequality only in the \emph{specific} cases 
of the vanishing-viscosity  and the stochastic approximation. 
\par
We believe that 
Theorem  \ref{thm:limsupgenerale}   could be extended to a broader class of dissipation potentials $\Psi_n$ with superlinear growth at infinity, like in the one-dimensional case (cf.\ \cite[Thm.\ 4.2]{BonaschiPeletier14}).
However, 
 the proof of the $\limsup$-estimate in the \emph{fully general} case, i.e.\ under the sole condition that the bipotentials $b_{\Psi_n}$ variationally converge to $\bptname$, remains an open problem.
\paragraph{\bf Plan of the paper.} In \underline{Section $2$} we discuss the multi-dimensional analogue of the stochastic model 
considered in \cite{BonaschiPeletier14}  and (formally) derive the associated dissipation potential $\Psi_n$ and the induced trajectory functional $\Jfuname{\Psi_n,\calE}$. 
\underline{Section \ref{s:3}} is devoted to some recaps on $\BV$ functions,  which are  preliminary to the introduction of the  \emph{extended}  notion of Balanced Viscosity solution to a rate-independent 
system $\RIS$ (cf.\ Definition \ref{def:BV}), with $\bptname$ a viscosity contact potential in the sense of Definition \ref{def:visc-cont-pot}. We conclude this section by enucleating some basic properties of  Balanced Viscosity solutions. 
In \underline{Section \ref{ss:3.3}} we address the generation of a  viscosity contact potential starting from a family $(\Psi_n)_n$ of dissipation potentials  with superlinear growth at infinity. Our main results, Theorems \ref{thm:liminf}  and \ref{thm:limsupgenerale}, are stated in \underline{Section \ref{s:main}} and proved throughout 
\underline{Section \ref{s:proofs}}. \\
 \paragraph{\bf  Acknowledgment.}  We would like to thank
 Mark Peletier   for suggesting this problem to us  and for his support. 
 We are deeply grateful to
  Giuseppe Savar\'e for several enlightening discussions, in particular on the results of Section 4.  


\section{\bf The stochastic origin of rate-independent systems}
\label{s:2}
\noindent 
In this section we briefly describe the 
 multi-dimensional 
extension of the one-dimensional stochastic model  for rate-independent evolution  
considered in \cite{BonaschiPeletier14}.
\par
 We consider a jump process $t\mapsto X^h(t)$ on a $d$-dimensional lattice, with
lattice spacing $\frac1h$. The evolution of the  process can be described as follows: Fix the origin as initial point.
If the process is at the position $x$ at time $t$, then it jumps in continuous time to its neighbours $(x\pm \frac1m \be_i)$ with rate $mr_i^{\pm}$,
for $i=1, \ldots, d$, where $(\be_1, \ldots, \be_d)$ is the basis of $\R^d$, 
cf.\
 Figure~\ref{fig:birth-death}. The jump rates depend on two parameters $\alpha$ and $\beta$, and on the 
  partial
 derivatives $ \rmD_i \E: = \rmD_{x_i} \E $ 
 of a smooth energy functional $\calE: [0,T]\times \R^d \to \R$, namely
 \begin{equation}
 \label{jump-rates}
 r^+_i (\bx,t) = \alpha e^{-\beta \rmD_i \E(\bx,t)) }, \quad r^-_i (\bx,t) = \alpha e^{\beta \rmD_i \E(\bx,t) ) } \qquad \text{for } i =1, \ldots, d. 
\end{equation}

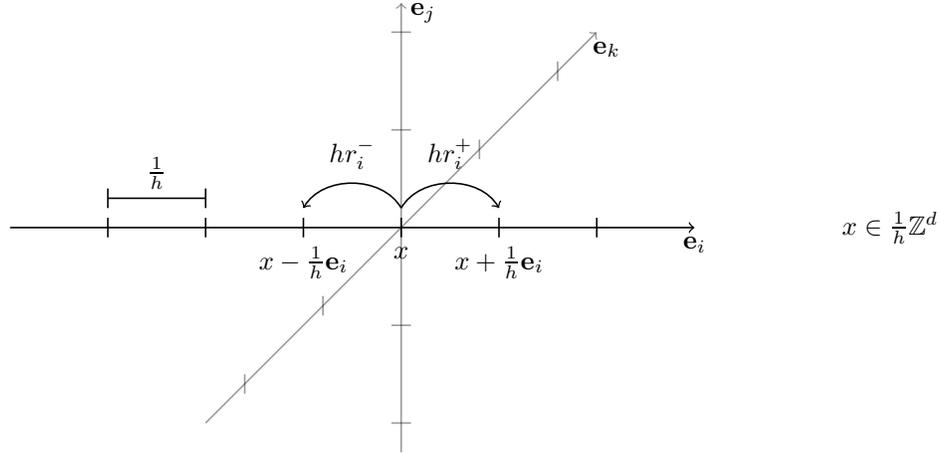
\begin{figure}[h]
\centering
\begin{tikzpicture}[scale=1.3]
\draw [->, semithick, opacity=.4] (-2,-2) -- (2,2);
\draw [semithick, opacity=.4] (.8,.7) -- (.8,.9);
\draw [semithick, opacity=.4] (1.6,1.5) -- (1.6,1.7);
\draw [semithick, opacity=.4] (-.8,-.7) -- (-.8,-.9);
\draw [semithick, opacity=.4] (-1.6,-1.5) -- (-1.6,-1.7);
\draw (2.1,2) node[anchor=north] {$\be_k$};
\draw [->, semithick, opacity=.4] (0,-2.3) -- (0,2.3);
\draw (0,2.2) node[anchor=west] {$\be_j$};
\draw [semithick, opacity=.4] (-.1,1) -- (.1,1);
\draw [semithick, opacity=.4] (-.1,2) -- (.1,2);
\draw [semithick, opacity=.4] (-.1,-1) -- (.1,-1);
\draw [semithick, opacity=.4] (-.1,-2) -- (.1,-2);
\draw [->, semithick, color=black] (-4,0) -- (3,0);
\draw (3,0) node[anchor=north] {$\be_i$};
\draw [semithick, color=black] (-3,.3) -- (-2,.3);
\draw [semithick, color=black] (-3,.4) -- (-3,.2);
\draw [semithick, color=black] (-2,.4) -- (-2,.2);
\draw (-1,-.1) node[anchor=north] {$x - \frac1h \be_i$};
\draw (0,-.1) node[anchor=north] {$x$};
\draw (1,-.1) node[anchor=north] {$x+\frac1h \be_i$};
\draw (-2.5,.3) node[anchor=south] {$\frac1h$};
\foreach \i in {-3,...,2}{
\draw [semithick, color=black] (\i,-.1) -- (\i,.1) ;}
\draw [->, semithick, color=black] (0,.2) to [out=120,in=60] (-1,.2);
\draw [->, semithick, color=black] (0,.2) to [out=60,in=120] (1,.2);
\draw (.5,1) node[anchor=north] {$h r_i^+$};
\draw (-.5,1) node[anchor=north] {$h r_i^-$};
\draw (5,0) node[anchor=center] {$x \in \frac1h \Z^d$};
\end{tikzpicture}
\caption{A sketch of the jump-process on the lattice.}
\label{fig:birth-death}
\end{figure}

The choice of  the stochastic process (and thus of the jump rates $r_i^\pm$)  reflects  Kramers' formula \cite{Kramers40,Berglund11TR,BonaschiPeletier14}. Given a particle evolving in a wiggly energy landscape with noise, this formula provides an estimate of the rate of jumps from  one energy well to the next one.
\par
We are interested in the continuum limit as $h\to\infty$. With this aim, we apply the method developed by 
\textsc{Feng \& Kurtz}, cf.\ \cite{FK06}, to prove large-deviations principles for Markov processes. 
\par
As in \cite[Sec.\ 2.5]{BonaschiPeletier14}, we will provisionally assume that the jump rates $r^\pm$ are constant in space and time, 
and thus derive the expression of the rate function, and then formally substitute \eqref{jump-rates} into it. 
Following \cite{FK06},  we consider the generator 
\[
\Omega_h f (\bx) := \sum_{i=1}^d \left[ \;h  r_i^+\left( f(\bx+ \frac1h \be_i)  - f(\bx) \right) + h r_i^-\left( f(\bx- \frac1h  \be_i) - f(\bx) \right) \right] 
\]
of the 
continuous time Markov process $X^h$, and the nonlinear generator
\[
\begin{aligned}
(\mathrm{H}_h f)(\bx): & = \frac1h e^{-h f(\bx)} (\Omega_h e^{hf}) (\bx)\\ &  = \sum_{i=1}^d \  \left[ \; r_i^+\left( \exp\left(h \left(  f(\bx+ \frac1h \be_i)  - f(\bx) \right) 
 \right)-1\right) +  r_i^-\left( \exp\left(h \left(  f(\bx- \frac1h \be_i)  - f(\bx) \right) 
 \right)-1\right) 
   \right]\,.
   \end{aligned}
\]
\par
According to the Feng-Kurtz method, 
if $\mathrm{H}_h$ converges to  some $\rmH$ in a suitable sense, and if
 the limiting operator 
 $\rmH f$ depends locally on $\rmD f$, we can then  define the \emph{Hamiltonian} $H = H(x,\xi)$
through
\[
(\mathrm{H} f)(x) = : H(x, \rmD f(x)),
\]
and the \emph{Lagrangian} as the Legendre transform  of $H$, namely
\[
L(x,v) : = \sup_{\xi\in\R^d} \left(\langle \xi, v \rangle - H(x,\xi) \right).
\]
Then, the Markov process satisfies a large-deviations principle, with rate function 
\begin{equation}
\label{Jfun-2}
\mathscr{J}(u):= \begin{cases}
\int_0^T L(u(t), u'(t)) \dd t & \text{if } u \in \AC ([0,T];\R^d),
\\
+\infty & \text{otherwise}, 
\end{cases}
\end{equation}
 cf.\ \cite[Sec.\ 2]{BonaschiPeletier14}. 
\par
In the present case,  it can be seen that 
 \[
 H(x,\xi)= \sum_{i=1}^d r_{i}^{+} (e^{\xi_i}-1) + r_{i}^{-} (e^{-\xi_i}-1).
 \]
Then $L$ is given by
\begin{equation}
\label{eq:defLagrangian}
\begin{split}
L(x,v) =
\sum_{i=1}^d \left[ v_i \log \left( \frac{v_i + \sqrt{v_i^2 + 4r^+_i r^-_i} }{2r^+_i} \right) - \sqrt{v_i^2 + 4r^+_i r^-_i} + r^+_i +r^-_i \right].
\end{split}
\end{equation}
Substituting in \eqref{eq:defLagrangian} 
the expression \eqref{jump-rates} for  the jump rates, and choosing the parameters
\[
\alpha = \frac{e^{- n A}}{2} \quad \text{and} \quad \beta =n, \qquad n \in \N,
\]
we obtain 
\begin{equation}
\label{our-lagra}
L(x,v) = n \left( \Psi_n (v) + \Psi_n^* (-\rmD \ene tx) +  v \rmD \ene tx \right),
\end{equation}
with $\Psi_n: \R^d \to [0,+\infty)$ given by 
\begin{equation}
\label{def:cosh_psi}
\Psi_n(\bv) = \sum_{i=1}^d \psi_n(v_i) =  \sum_{i=1}^d \frac{v_i}{n} \log \left( \frac{ v_i + \sqrt{v_i^2 + e^{-2n A}}}{e^{-n A}} \right) - \frac{1}{n}\sqrt{v_i^2 + e^{-2n A}} + \frac{e^{-n A}}{n},
\end{equation}
 and  $\Psi_n^*$ the Legendre transform of $\Psi_n$. It can be easily checked that the structure $\Psi_n(\bv) = \sum_{i=1}^d \psi_n(v_i)$
 transfers to the conjugate, hence
\begin{equation}
\label{def:cosh_psi*}
\Psi^*_n(\xi) = \sum_{i=1}^d \psi^*_n (\xi_i) = \sum_{i=1}^d \frac{e^{-n A}}{n} \left( \cosh ( n \xi_i ) - 1 \right).
\end{equation}
\par
Observe that,
with the choice \eqref{our-lagra} for $L$,  the (positive)
 functional $\mathscr{J}$
 from \eqref{Jfun-2} is minimized by the solutions of the ODE system
  (the subdifferential operator $\partial\psi_n^* :\R^d \rightrightarrows \R^d $ being single-valued) 
\[
u_i'(t) = - \mathrm{D} \psi^*_n (-\rmD_i  \calE(t, u_i(t)))  \qquad \foraa\,t \in (0,T), \text{  for all } i=1, \cdots , d.
\]

\section{\bf Viscosity contact potentials and Balanced Viscosity solutions to rate-independent systems}
\label{s:3}
\RRR This section is devoted to the notion of \emph{Balanced Viscosity} solution to a rate-independent system. Before introducing it and fixing its main properties, 
we recall some basic definitions from the theory of functions of bounded variation, and then focus on the crucial concept of \emph{viscosity contact potential}, which underlies the very definition of 
\emph{Balanced Viscosity} solution. \EEE
\subsection{Preliminary definitions}
\label{ss:3.1}
Hereafter, we will call
\emph{dissipation potential} any function
\begin{equation}
\label{dissip-potential}
\Psi: \R^d \to [0,+\infty)
\text{ convex and such that  } \Psi(0)=0.
\end{equation}
It follows from the above conditions that the Fenchel-Moreau conjugate $\Psi^*$ then fulfills $\Psi^*(0)=0 \leq \Psi^*(\xi)$ for all $\xi\in\R^d$.
\RRR We will distinguish two cases:
\begin{itemize}
\item[{\bf Dissipation potentials with superlinear growth at infinity}] i.e.\ fulflling 
\begin{equation}
\label{superlinear-infty} 
\lim_{\|v\|\to+\infty} \frac{\Psi(v)}{\|v \|} = +\infty
\end{equation}
for some norm $\| \cdot\|$ on $\R^d$. \EEE
\item[{\bf  $1$-homogeneous dissipation potentials}]
In what follows, we will denote by
$\psiri$
a dissipation potential
\begin{equation}
\label{rate-indep-pot} \psiri : \R^d \to [0,+\infty) \
\text{convex, $1$-positively homogenous, and non-degenerate, viz. } \psiri(\bv) >0 \ \text{ if } \bv \neq 0.
\end{equation}
Thus, for any norm $\| \cdot \|$ on $\R^d$
\begin{equation}
\label{equivalent}
\exists\, \eta>0 \ \forall\, \bv \in \R^d \, : \ \eta^{-1} \|\bv\| \leq \psiri(\bv) \leq \eta \|
\bv\|\,.
\end{equation}
Its convex-analysis subdifferential $\partial \psiri : \R^d
\rightrightarrows \R^d$  at $\bv \in \R^d$ can be characterized by
\begin{equation}
\label{charact-1homog}
\zzeta \in \partial\psiri(\bv) \ \Leftrightarrow \ \left\{
\begin{array}{ll}
\langle \zzeta, \bw \rangle \leq \psiri(\bw)  \text{ for all } \bw
\in \R^d,
\\
\langle \zzeta, \bv \rangle = \psiri(\bv).
\end{array}
\right.
\end{equation}
Throughout, we will use the notation
\begin{equation}
\label{Kstar} K^*: = \partial \psiri(0)\,.
\end{equation}
Recall that $\partial\psiri(\bv) \subset K^*$ for all $\bv \in \R^d$
and that, indeed, $\psiri$ is the \emph{support function} of $K^*$,
namely
\begin{equation}
\label{support-function} \psiri(\bv) = \sup_{\zzeta \in K^*} \langle
\zzeta, \bv \rangle, \quad \text{whence } \quad \psiri^*(\xi) =
I_{K^*}(\xi).
\end{equation}
\end{itemize} \EEE
\paragraph{\bf $\BV$ functions}
Throughout, we will work with
functions of bounded variation \emph{pointwise defined
  at every point $t\in [0,T]$}.
  We recall that   a function $u$ in $\BV([0,T];\R^d)$ admits left and right limits at every  $t\in [0,T]$:
  \begin{equation}
    \label{eq:41}
    u(t_-):=\lim_{s\up t}u(s),\ \
    u(t_+):=\lim_{s\down t}u(s),\ \ \text{with the convention }u(0_-):=u(0),\ u(T_+):=u(T),
  \end{equation}
  and its \emph{pointwise} jump set $\ju u$ is the at most countable set defined
  by
  \begin{equation}
    \label{eq:88}
    \ju u:=\big\{t\in [0,T]:u(t_-)\neq u(t)\text{ or }u(t)\neq u(t_+)\big\}\supset
    \essJ_u:=
    \big\{t\in [0,T]:u(t_-)\neq u(t_+)\big\}.
  \end{equation}
 We also recall that the distributional derivative $u'$ of $u$ is a Radon vector measure that can be decomposed (cf.\ \cite{Ambrosio-Fusco-Pallara00})
  into the
sum of the three mutually singular measures
\begin{equation}
  \label{eq:87}
  u'=u'_\Le+u'_\Ca+u'_{\mathrm{J}},\quad u'_\Le=\dot u\,\mathscr{L}^1,\quad
  u'_\co:=u'_\Le+u'_\Ca\,.
\end{equation}
Here,  $u'_\Le$ is the absolutely continuous part with
respect to the Lebesgue measure $\mathscr{L}^1$, whose Lebesgue density
$\dot u$ is the  pointwise (and $\Leb 1$-a.e.\ defined)
derivative of $u$, $u'_\mathrm{J}$ is a discrete measure concentrated on
$ \essJ_u \subset \ju u$, and $u'_\Ca$ is the so-called Cantor part. We will use the notation
$u'_\co:=u'_\Le+u'_\Ca$ for  the diffuse part of the measure, which
does not charge $ \ju u$.

Given  a (non-degenerate) $1$-homogeneous dissipation potential $\psiri$, it induces a notion of (pointwise) total variation for a curve $u\in \BV ([0,T];\R^d)$ via
 \begin{equation}
 \label{varpsi}
    \Var{\psiri}uab:=
    \sup\Big\{\sum_{m=1}^M\psiri\big(u(t_m)-u(t_{m-1})\big):a=t_0< t_1<\cdots<t_{M-1}<t_M=b\Big\}
    \end{equation}
    for any  $[a,b]\subset [0,T]$.
    Therefore, with any $u\in \BV ([0,T];\R^d)$ we can associate the non-decreasing function $V_{\psiri}: \R \to [0,+\infty)$
    defined by
    \[
    V_{\psiri}(t): = \left\{
    \begin{array}{lll}
    0  & \text{if } t \leq 0,
    \\
    \Var{\psiri}{u}0t & \text{if } t \in (0,T),
    \\
    \Var{\psiri}{u}0t & \text{if } t\geq T.
    \end{array}
    \right.
    \]
    Its distributional derivative $\mu_{\psiri}$ is in turn a Radon measure that can be decomposed into a jump part
    $\mu_{\psiri,\mathrm{J}}$, concentrated on $\ju u$ and
    given by
    \[
    \mu_{\psiri,\mathrm{J}} (\{t\}) =
   \psiri (u(t)-u(t_-)) + \psiri (u(t_+)-u(t)),
    \]
     and a diffuse part
     \begin{equation}
     \label{mu-psiri-co}
     \mu_{\psiri,\co} = \mu_{\psiri,\Leb{}} +  \mu_{\psiri,\Ca}  \qquad \text{with}\quad  \mu_{\psiri,\Leb{}}  = \psiri(\dot u) \Leb 1.
     \end{equation}
      There holds
     \begin{equation}
     \label{implicit-def-mu}
      \Var{\psiri}uab = \mu_{\psiri,\co} ([a,b]) + \Jump {\psiri}{u}{a}{b},
     \end{equation}
     with the jump contribution $ \Jump {\psiri}{u}{a}{b}$  given by
     \begin{equation}
     \label{jump-psiri}
     \begin{aligned}
   &  \Jump {\psiri}{u}{a}{b}:   =
     \psiri(u(a_+)-u(a)) + \mu_{\psiri,\mathrm{J}} ((a,b)) +   \psiri(u(b_+)-u(b))
     \\
     &\quad
      = \psiri(u(a_+)-u(a)) + \sum_{t\in \ju u  \cap (a,b)} \Big(    \psiri (u(t)-u(t_-)) + \psiri (u(t_+)-u(t))  \Big) +   \psiri(u(b_+)-u(b)) \,.
      \end{aligned}
     \end{equation}

Finally,
for later use
 we recall that a sequence $(u_n)_n $ \emph{weakly} converges in $ \BV ([0,T];\R^d)$ to a curve $u$ (we will write $u_n \weakto u$)
if $u_n(t) \to u(t)$ as $n\to\infty$ for every $t\in [0,T]$ and
$\sup_n \Var{}{u_n}{0}T \leq C<\infty$ \RRR (in what follows, we shall denote by $\Var{}{u}{0}{T}$ the total variation of a curve $u$ induced by a generic norm on $\R^d$), \EEE whereas
 $(u_n)_n $ \emph{strictly} converges in $ \BV ([0,T];\R^d)$ to $u$ ($u_n\to u$)  if  $u_n \weakto u$ and $\Var{}{u_n}{0}T  \to \Var{}{u}{0}T $.
 %
\paragraph{\bf Viscosity contact potentials.}
The notion we are going to introduce now lies at the core of the
definition of \emph{Balanced Viscosity} solution to a
rate-independent system, driven by an energy functional $\calE$
complying with \eqref{smooth-energy}. Indeed, the concept of
\emph{viscosity contact potential}
 encodes how viscosity enters into the description of the solution behavior at jumps.
It is an extension of the notion of \emph{vanishing-viscosity
contact potential} introduced in \cite{MRS10}, in that  we  are
augmenting the contact  potential defined therein by the time
variable. 
 In referring to this notion, we will drop the word `vanishing'
in order to highlight that Balanced Viscosity solutions do not
necessarily arise from a vanishing-viscosity approximation, cf.\
Sec.\ \ref{s:example}.
\begin{definition}
\label{def:visc-cont-pot}
We call a
 lower semicontinuous  
 function $\bptname: [0,+\infty) \times \R^d \times \R^d \to  \RRR [0,+\infty] \EEE $ \emph{(viscosity) contact potential} if  it satisfies the following properties:
\begin{enumerate}
\item for every $\tau\geq0$ there holds
$\bpt{\tau}{\bv}{\bxi} \geq \langle  \bv, \bxi \rangle $ for all $(\bv,\bxi) \in \R^d \times \R^d$;
\item for every $\bxi \in \R^d$
the map $(\tau,\bv) \mapsto \bpt {\tau}{\bv}{\bxi}$ is convex and positively $1$-homogeneous.
 \item for every $\tau>0 $  and $\bv \in \R^d$,  the map $\xi \mapsto \bpt{\tau}{\bv}{\xi}$ 
is convex. 
\end{enumerate}
\RRR Moreover, we say that $\bptname$ is \emph{non-degenerate}
 if 
 \begin{enumerate}
 \setcounter{enumi}{3}
 \item for every $\tau \geq 0$  there holds  $\bpt{\tau}{\bv}{\bxi} >0$ if $\bv \neq 0$. 
 \end{enumerate} \EEE
Finally, given a (non-degenerate) $1$-homogeneous dissipation potential $\psiri$ as in \eqref{rate-indep-pot},
 we say that
$\bptname$ is \emph{$\psiri$-non degenerate}  if
\begin{enumerate}
 \setcounter{enumi}{4}
 \item  for all $(v,\bxi) \in \R^d \times \R^d$ there holds  $\bpt{0}{\bv}{\bxi} \geq \psiri(\bv)$.
\end{enumerate}
\end{definition}

A crucial object   related to
a  (viscosity) contact potential $\bptname$
 is the set where the inequality in (2) holds as an equality. We will  call  it \emph{contact set} and denote it by 
 \begin{equation}
 \label{def-contact-set}
\Ctc{\bptname} := \left\{ (\tau,\bv,\bxi) \in [0,+\infty) \times \R^d \times \R^d  \, :  \, \bpt{\tau} {\bv}{\bxi}  =  \langle \bv,\bxi \rangle  \right\},
\end{equation}
whereas we will use the notation
 \begin{equation}
 \label{Lambdap0}
 \Ctc{\bptname,0} : = \Ctc{\bptname} \cap  \{0 \} \times  \R^d \times \R^d = \{ (\bv,\bxi) \in \R^d \times \R^d \, : \,   \bpt0 \bv{\bxi}  =  \langle \bv,\bxi \rangle \}.
\end{equation}
Let us point out a first important consequence of the properties defining a contact potential:
\begin{lemma}
\label{l:charact-cont-set}
For  fixed $(\tau,\bxi) \in [0,+\infty) \times \R^d$, denote by
$\partial_v \bptname (\tau,\cdot, \bxi)(\bv)$ the
  (convex analysis) subdifferential at $\bv$ of the functional $\bv \mapsto \bpt \tau {\bv} {\bxi}$.
Then,
\begin{equation}
\label{charact}
(\tau,\bv,\bxi) \in \Ctc{\bptname} \ \Leftrightarrow \  \bxi \in \partial_v \bpt \tau{\cdot}\bxi(\bv)
\end{equation}
\end{lemma}
\begin{proof}
Since $\bv \mapsto \bpt \tau \bv \bxi$ is convex and positively homogeneous of degree $1$, we have (cf.\ \eqref{charact-1homog}),
\[
\bxi \in \partial_v \bpt \tau{\cdot}{\bxi}(\bv) \quad \text{iff} \quad
\begin{cases}
\langle \bxi, \tilde \bv \rangle \leq \bpt \tau{\tilde {\bv}}\bxi
\quad \text{for all } \tilde{v} \in \R^d,
\\
\langle  \bxi, \bv \rangle = \bpt  \tau{\bv}{\bxi},
\end{cases}
\]
and the thesis follows.
\end{proof}
\begin{remark}
\label{not-a-bipot}
\upshape
Observe that, for fixed  $\tau \in [0,+\infty)$,  the  function
$\bpt{\tau}{\cdot}{\cdot}: \R^d \times \R^d \to [0,+\infty)$ enjoys
    some of the properties of the notion of \emph{bipotential} (cf., e.g., \cite{Buliga-deSaxce-Vallee08}), which is by definition a functional
    $\mathsf{b} : \R^d \times \R^d \to [0,+\infty]$
     convex and lower semicontinuous w.r.t.\  \emph{both} variables, separately, and fulfilling
     $\mathsf{b}(\bv,\bxi) \geq \langle \bv,\bxi \rangle$ for all $(\bv,\bxi) \in \R^d \times \R^d$, as well as
     a stronger version of  \eqref{charact}, namely
     \[
     (\bv,\bxi) \in \Ctc{\mathsf{b}} \  \Leftrightarrow \ \bxi \in \partial_v \mathsf{b}(\cdot,\bxi)(\bv) \  \Leftrightarrow \ \bv \in \partial_\xi \mathsf{b}(\bv,\cdot)(\bxi)\,,
\]
where the contact set $\Ctc{\mathsf{b}} $ is defined similarly as in \eqref{def-contact-set}.
\par
As discussed in \cite{MRS10},
 the conditions defining the notion of bipotential seem to be too restrictive
for the  contact potentials arising in the vanishing-viscosity limit  of viscous systems approximating rate-independent evolution.
Nonetheless, in Sec.\ \ref{ss:3.3} we will see how \emph{viscosity contact potentials} can in fact be generated, via $\Gamma$-convergence, by
bipotentials associated with  families of  dissipation potentials.
\end{remark}

\subsection{$\BV$ solutions to rate-independent systems}
\label{ss:bv}
We are now in a position to recall   the preliminary definitions at the basis of the concept of
\emph{Balanced Viscosity} solution;  notice that all of them  involve the \emph{reduced} contact potential $\bpt 0{\cdot}{\cdot}$ and the energy functional $\calE \in \rmC^1 ([0,T]\times \R^d)$.

First of all, we introduce the (possibly asymmetric) Finsler distance coming into play in the description of the energetic behaviour of a rate-independent  system at a jump time: For a fixed $t\in [0,T]$, the Finsler distance induced by $\bptname$
and $\calE$ at the time $t$ is defined for every $ u_0,\, u_1\in \R^d$ by
\begin{equation}
\label{Finsler-cost}
 \Cost{\bptname}{\calE}{t}{u_0}{u_1}:=
\inf\left\{ \int_{r_0}^{r_1} \bpt{0}{\dot{\theta}(r)}{-\rmD \ene
t{\theta(r)}} \dd r   \, : \ \theta \in \mathrm{AC}
([r_0,r_1];\R^d), \ \theta(r_0) = u_0, \,    \theta(r_1) = u_1
\right\}.
\end{equation}
Observe that,  if $\bptname$  is  a $\psiri$-non degenerate contact potential  for some $1$-positively homogeneous potential $\psiri$, we clearly have
$\Cost{\bptname}{\calE}{t}{u_0}{u_1} \geq \Delta_{\psiri}(u_0,u_1): = \psiri (u_1-u_0)$.
The Finsler distances  from \eqref{Finsler-cost} induce a notion of total variation that measures the dissipation of a $\BV$-curve at its jump points,
mimicking the notion \eqref{varpsi} of $\psiri$-total variation.  Namely, along the footsteps of
\cite[Def.\ 3.4]{MRS10} and in analogy with \eqref{jump-psiri},  for a given curve $u\in \BV ([0,T];\R^d)$ with jump set $\ju u $,  we define the \emph{jump variation} of $u$ induced by
$(\bptname,\calE)$ on an interval $[a,b]\subset [0,T]$ by
\begin{equation}
\label{jump-p-E}
\begin{aligned}
 \Jump {\bptname,\calE}{u}{a}{b} : =
 & \Cost{\bptname}{\calE}{a}{u(a)}{u(a_+)} \\ & + \sum_{t\in \ju u  \cap
(a,b)}  \left(  \Cost{\bptname}{\calE}{t}{u(t_-)}{u(t)} +
\Cost{\bptname}{\calE}{t}{u(t)}{u(t_+)} \right)     +
\Cost{\bptname}{\calE}{b}{u(b_-)}{u(b)} \,.
\end{aligned}
\end{equation}
Finally, given a (non-degenerate) $1$-positively homogeneous  dissipation potential $\psiri$ and  a contact viscosity potential $\bptname$, 
the (pseudo-)total variation of a curve  $u \in \BV([0,T];\R^d)$ induced by $(\psiri, \bptname, \calE)$ is
defined  by (cf.\ \eqref{implicit-def-mu})
\begin{equation}
\label{pseudo-Var}
\Var{\psiri,\bptname,\calE}{u}ab:= \mu_{\psiri,\co} ([a,b]) + \Jump {\bptname,\calE}{u}{a}{b}\quad \text{for any $[a,b]\subset [0,T]$,}
\end{equation}
 with $ \mu_{\psiri,\co} $ from \eqref{mu-psiri-co} the diffuse part of the total variation measure  of the map $t\mapsto \Var{\psiri}{u}{0}t$.
 Let us mention  that the notation $\mathrm{Var}_{\psiri,\bptname,\calE}$ is used here with slight abuse,
since $\mathrm{Var}_{\psiri,\bptname,\calE}$ does not enjoy all of the standard properties of total variation functionals, see \cite[Rmk.\ 3.6]{MRS10} for further details.
Also observe that, if $\bptname$ is $\psiri$-non degenerate, then we have $\Var{\psiri,\bptname,\calE}{u}ab \geq \Var{\psiri}{u}ab$.

We are finally in a  position  to recall the concept of \emph{Balanced Viscosity} solution, cf.\ \cite[Def.\ 4.1]{MRS10} and \cite[Def.\ 3.10]{mielke-rossi-savare2013}.
\begin{definition}[Balanced Viscosity solution]
\label{def:BV}
Given a (non-degenerate) $1$-homogeneous dissipation potential $\psiri$ and a (non-degenerate) viscosity contact potential $\bptname$, we say that a curve
$u\in \BV ([0,T];\R^d)$ is a  \emph{Balanced Viscosity ($\BV$) solution to the rate-independent system $\RIS$} if it fulfills the \emph{local stability} \eqref{eq:65bis} and the
\eqref{eq:84}-\emph{energy balance}
  \begin{equation}
    \label{eq:65bis}
    \tag{$\mathrm{S}_\mathrm{loc}$}
    -\rmD\ene t{u(t)}\in K^*\quad \text{for all }\ t \in [0,T] \setminus \ju u,
  \end{equation}
  \begin{equation}
    \label{eq:84}
    \Var{\psiri,\bptname,\calE}u{0}{t}+\ene{t}{u(t)}=\ene{0}{u(0)}+
    \int_{0}^{t} \partial_t\ene s{u(s)}\,\dd s \quad \text{ for all } t\in [0,T]
    \tag{$\mathrm{E}_{\psiri, \bptname,\calE}$}
  \end{equation}
  with $K^* = \partial\psiri (0)$.
\end{definition}

While referring  to  \cite[Sec.\ 4]{MRS10} and \cite[Sec.\ 3]{mielke-rossi-savare2013} for a detailed survey of the properties of $\BV$ solutions, let us only mention here that this concept yields a thorough description of the energetic behavior of the solution at jumps through
the concept of  \emph{optimal jump  transition}. For fixed $t\in [0,T]$ and $u_-, u_+ \in \R^d$,  we call a curve  $ \theta \in \mathrm{AC} ([0,1];\R^d)$
(up to a rescaling, we may indeed suppose the curves in \eqref{Finsler-cost} to be defined on $[0,1]$),
 with  $ \theta(0) = u_- $ and     $\theta(1) = u_+$, a $(\bptname, \calE_t )$-optimal transition between $u_-$ and $u_+$ if
\begin{equation}
\label{optimal-jump-transition}
\ene t{u_-} - \ene t {u_+} = \Cost{\bptname}{\calE}{t}{u_-}{u_+} =  \bpt{0}{\dot{\theta}(r)}{{-}\mathrm{D}\calE(t,\theta(r))} >0 \qquad \foraa\, r \in (0,1). 
\end{equation}
The following result subsumes
\cite[Prop.\ 4.6, Thm.\ 4.7]{MRS10}.
\begin{proposition}
\label{prop:OJT}
Let $u\in \BV ([0,T];\R^d)$ be a Balanced Viscosity solution to the rate-independent system $\RIS$. Then, at every jump time $t\in \ju u $
there exists  a   $(\bptname,\calE_t)$-optimal transition   $\theta^t$ between
 the left and right-limits $u_-(t)$ and $u_+(t)$, such that
 $\theta^t (r) = u(t)$ for some $r\in [0,1]$.
 Moreover, any optimal jump transition $\theta^t$
between $u_-(t)$ and $u_+(t)$
 complies with  the   contact contact condition
 \begin{equation}
 \label{contact}
 (\dot{\theta}^t(r), -\rmD \ene t{\theta^t(r)}) \in \Ctc{\bptname,0} \qquad \foraa\, r \in (0,1),
 \end{equation}
  with $ \Ctc{\bptname,0} $ from \eqref{Lambdap0}. 
\end{proposition}
A crucial consequence of
\eqref{contact} and of \eqref{charact} from Lemma \ref{l:charact-cont-set} is that any optimal jump transition $\theta^t$ complies with the subdifferential inclusion
\begin{equation}
\label{subdiff-inclusion}
-\rmD \ene t{\theta^t(r)} \in \partial_v \bpt 0{\cdot}{-\rmD \ene t{\theta^t(r)}}(\dot{\theta}^t(r)) \qquad \foraa\, r \in (0,1).
\end{equation}
This explicitly  highlights how the  contact potential $\bptname$   enters into the   description of the solution behavior at jumps.

With the last result of this section we reformulate the $\BV$ solution concept in terms of the null-minimization of a functional defined on $\BV$-trajectories; this will be crucial
for the variational convergence analysis developed in Sec.\ \ref{s:main}. Namely, given a  rate-independent system $\RIS$, we
define  the functional $\Jfuname{\psiri,\bptname,\calE}:   \BV ([0,T];\R^d) \to (-\infty,+\infty]$
 by
\begin{equation}
\label{J-fun_RIS}
\begin{aligned}
\Jfu{\psiri,\bptname,\calE}u:  & =
\Var{\psiri,\bptname,\calE}{u}{0}T +  \int_0^T  \psiri^* (-\rmD \ene t{u(t)}) \dd t   +\ene{T}{u(T)} - \ene{0}{u(0)} -     \int_{0}^{T} \partial_t\ene s{u(s)}\,\dd s
\\
&
\begin{aligned}
 =
 \int_0^T
 \psiri(\dot{u}(s)) &  + \psiri^* (-\rmD \ene s{u(s)}) \dd s    + \mu_{\psiri,\Ca}([0,T])   + \Jump{\bptname,\calE}{u}{0}T\\ &
+\ene{T}{u(T)} - \ene{0}{u(0)} -     \int_{0}^{T} \partial_t\ene s{u(s)}\,\dd s.
\end{aligned}
\end{aligned}
\end{equation}
We then have the following
\begin{proposition}
\label{prop:null-minim}
A curve $u\in \BV ([0,T];\R^d)$ is a   Balanced Viscosity solution to the rate-independent system $\RIS$ if and only if
\begin{equation}
\label{null-minimizer}
0= \Jfu{\psiri,\bptname,\calE}u \leq  \Jfu{\psiri,\bptname,\calE}v \qquad \text{for all } v \in  \BV ([0,T];\R^d)\
\end{equation}
\end{proposition}
\begin{proof}
First of all, observe that
 conditions \eqref{eq:65bis}--\eqref{eq:84} are indeed equivalent to \eqref{eq:65ter}--\eqref{eq:84}, with
\begin{equation}
    \label{eq:65ter}
    \tag{$\mathrm{S}_\mathrm{loc}'$}
    -\rmD\ene t{u(t)}\in K^*\quad \text{for } \Leb1-\text{a.a. } t \in (0,T).
  \end{equation}
  Indeed,
  if \eqref{eq:65ter} holds, with a
   continuity argument one deduces $ -\rmD\ene t{u(t)}\in K^*$ at all $t\in [0,T]\setminus \ju u$.

Clearly,  \eqref{eq:65ter}--\eqref{eq:84} are then equivalent to
  \begin{equation}
\label{crucial-rif}
  \Jfu{\psiri,\bptname,\calE}u =0. 
  \end{equation}
  Now, with an argument based on the chain rule for $\calE$, one sees (cf.\   the proof of  \cite[Cor.\ 3.4]{mielke-rossi-savare2013}) that  along a given  curve  \RRR $v\in \BV ([0,T];\R^d)$
  the map
  $ \Jfu{\psiri,\bptname,\calE}v \geq 0$, so that
  \eqref{crucial-rif}  holds if and only if $  \Jfu{\psiri,\bptname,\calE}u \leq 0$,  i.e.\ 
  $u\in \mathrm{Argmin}_{v\in   \BV ([0,T];\R^d) } \Jfu{\psiri,\bptname,\calE}v$.  \EEE
   This concludes the proof.
  \end{proof}
\section{Generation of viscosity contact potentials via $\Gamma$-convergence}
\label{ss:3.3}
In this section we show a possible way to generate a viscosity contact potential via a
 $\Gamma$-convergence procedure, starting from a family $(\Psi_n)_n$
of dissipation potentials  \RRR  with \emph{superlinear growth at infinity} (cf.\ \eqref{superlinear-infty}).  \EEE

Preliminarily, given a convex dissipation potential $\Psi$, we define the  \emph{bipotential}
$ \bipname{\Psi} : [0,+\infty) \times \R^d \times \R^d \to [0,+\infty]$
induced by $\Psi$ via
\begin{equation}
\label{def-bptreg}
\bip{\Psi}{\tau}{v}{\xi}:=
 \left\{
 \begin{array}{ll}
 \tau \Psi \left( \frac v\tau \right) + \tau \Psi^*(\xi)  &  \text{for } \tau>0,
 \\
 0   & \text{for } \tau=0, \ v=0,
 \\
 +\infty  & \text{for } \tau =0 \text{ and } v \neq 0,
 \end{array}
 \right.
  =
 \left\{
 \begin{array}{ll}
 \tau \Psi \left( \frac v\tau \right) + \tau \Psi^*(\xi)  &  \text{for } \tau>0,
 \\
I_{\{ 0\}}(v)   & \text{for } \tau=0.
 \end{array}
 \right.
 \end{equation}
 It is immediate to check that
 \begin{enumerate}
 \item for every $(v,\xi)\in \R^d \times \R^d$ the map $\tau\mapsto \bip{\Psi}{\tau}{v}{\xi}$ is convex;
 \item for every $\tau \geq 0$  the functional $(v,\xi)\mapsto \bip{\Psi}{\tau}{v}{\xi}$ is a bipotential in the sense of \cite{Buliga-deSaxce-Vallee08} (cf.\ Remark \ref{not-a-bipot});
 \item \RRR for every $v \neq 0$ and $\xi \in \R^d$ with $\Psi^*(\xi) \neq 0$, the set $\mathrm{Argmin}_{\tau>0} \bip{\Psi}{\tau}{v}{\xi}$ is non-empty, 
 \end{enumerate}
\RRR where the latter property is due to the fact that
  $\lim_{\tau\downarrow 0}  \bip{\Psi}{\tau}{v}{\xi} =+\infty $ due to the superlinear growth of $\Psi$, and 
  $\lim_{\tau\uparrow+\infty}  \bip{\Psi}{\tau}{v}{\xi} =+\infty $. \EEE

 Let us now be given a sequence $(\Psi_n)_n$ of dissipation potentials, and let $(\bipname{\Psi_n})_n$ be the associated bipotentials.
 We assume the following.
 \begin{hypothesis}
\label{hyp:p-lim} \upshape Let $  \bptname : [0,+\infty) \times \R^d
\times \R^d \to [0,+\infty]$   be  defined by
\begin{equation}
\label{def-p-gliminf}
\begin{aligned}
&
 \bptname= \gl \liminf_{n} \bipname{\Psi_n} \quad \text{ i.e. } \quad
\bpt \tau v \xi : = \inf\{ \liminf_{n \to \infty} \bip{\Psi_n}{\tau_n}{v_n}{\xi_n} \, : \ \tau_n \to \tau, \quad v_n \to v \quad \xi_n \to \xi \}.
\end{aligned}
\end{equation}
Then,
\begin{equation}
\label{hyp-p-limsup}
\begin{aligned}
 & \text{for every } \xi \in\R^d  \text{ there exists $(\xi_n)_n \subset \R^d$ with }   \xi_n \to\xi \text{  and  }
 \bpt{\cdot}{\cdot}{\xi}  = \gl \limsup_{n \to \infty} \bip{\Psi_n} {\cdot}{\cdot}{\xi_n} \qquad \text{i.e. }
\\
&
\bptnew \tau v {\xi} =   \inf_{(\xi_n)_n\subset \R^d,\, \xi_n\to \xi}  \left\{ \limsup_{n \to \infty} \bip {\Psi_n}{\tau_n}{v_n}{\xi_n} \, : \ \tau_n \to \tau, \quad v_n \to v \right\}.
\end{aligned}
\end{equation}
\end{hypothesis}
In Section \ref{s:example} ahead, we will exhibit two classes of dissipations potentials $(\Psi_n)_n$, with superlinear growth at infinity,  and associated functionals $\bptname$, complying with Hypothesis \ref{hyp:p-lim}. 
\par
Observe that with \eqref{hyp-p-limsup}
 we are imposing a stronger condition than $\bptname = \gl \limsup_{n\to\infty} \bipname{\Psi_n}$,
 namely we are asking that 
 \begin{equation}
 \label{gamma-limsup-cond}
 \forall\, \xi \in \R^d \ \ \exists\, (\xi_n)_n \subset \R^d\, :  
 \ \ \forall\, (\tau,v) \in   [0,+\infty) \times \R^d \ \exists\, (\tau_n,v_n)_n\, \text{ s.t. } 
 \begin{cases}
 \tau_n\to\tau, \\ v_n \to v, \\  \limsup_{n \to \infty} \bip {\Psi_n}{\tau_n}{v_n}{\xi_n} \leq \bptnew \tau v {\xi}\,.
 \end{cases}
 \end{equation}
  This property will play a \RRR key \EEE role in the proof of Lemma \ref{l:3} below.

The main result of this section ensures that the functional $\bptname$ generated  via \eqref{def-p-gliminf}--\eqref{hyp-p-limsup} is a contact potential in the sense of Definition \ref{def:visc-cont-pot}.
\begin{theorem}
\label{thm:generation}
Let $(\Psi_n)_n$ be a sequence of dissipation potentials on $\R^d$ complying with Hypothesis \ref{hyp:p-lim}. Then, $\bptname$ is a  viscosity contact potential  according to Def.\  \ref{def:visc-cont-pot}, 
and there exists a $1$-homogeneous dissipation potential $\psiri$ such that
\begin{equation}
\label{psiri-non-deg}
\bpt \tau v\xi \geq \psiri(v) \qquad  \text{for all } (\tau,v,\xi) \in [0,+\infty) \times \R^d \times \R^d.
\end{equation}
\par
\RRR Moreover, if the dissipation potentials $(\Psi_n)_n$ fulfill
\begin{equation}
\label{4-non-degeneracy}
\exists\, M>0, \ (M_n)_n \subset (0,+\infty) \text{ s.t. } M_n \to 0 
 \ \text{and} \  \forall\, n \in \N \ \ \forall\, v \in \R^d \text{ there holds } \qquad \Psi_n(v) \geq M \| v \|-M_n,
\end{equation}
then $\psiri$ is non-degenerate, and thus $\bptname$ is $\psiri$-non degenerate. \EEE
\end{theorem}
We postpone the proof  of Theorem \ref{thm:generation} to the end of this section, after obtaining a series of preliminary lemmas on the structure that $\bptname$ 
defined by Hypothesis \ref{hyp:p-lim} 
inherits from the potentials $\Psi_n$.
\begin{lemma}
\label{l:3} Assume Hypothesis \ref{hyp:p-lim}. Then,  for every
$(\tau,v,\xi) \in [0,+\infty) \times \R^d \times \R^d$ there holds
\begin{enumerate}
\item $\bptnew \tau v\xi  \geq \langle v,\xi \rangle $;
\item the map $(\tau,v) \mapsto \bpt \tau v\xi$ is convex and positively homogeneous of degree $1$.\EEE
\end{enumerate}
\end{lemma}
\begin{proof}  \RRR Property \EEE
(1) is an immediate consequence of \eqref{def-p-gliminf}, using that
for every $n\in \N$ there holds $\bip{\Psi_n}{\tau}{v}{\xi} \geq \langle v , \xi \rangle$
for every $(\tau,v,\xi) \in [0,+\infty) \times \R^d \times \R^d$.
\par
As for (2),
for fixed $\xi$ let $(\xi_n)_n$ fulfill \eqref{hyp-p-limsup}.  For fixed $(\tau_0,v_0)$ and $(\tau_1,v_1)$ let $(\tau_n^i,v_n^i)_n$, $i=1,2$, be two associated  recovery sequences for
$\bip{\Psi_n}{\cdot}{\cdot}{\xi_n}$
 as in \eqref{gamma-limsup-cond}. 
 Then,  for every $\lambda \in [0,1]$ there holds
\[
\begin{aligned}
\bpt{(1-\lambda)\tau_0 + \lambda \tau_1}{(1-\lambda)v_0 + \lambda v_1}{\xi} &  \stackrel{(1)}{\leq} \liminf_{n\to\infty}
\bip{\Psi_n}{(1-\lambda)\tau_n^0 + \lambda \tau_n^1}{(1-\lambda)v_n^0 + \lambda v_n^1}{\xi_n}
\\ &  \stackrel{(2)}{\leq}  \limsup_{n\to\infty}
 (1-\lambda) \bip{\Psi_n}{\tau_n^0}{v_n^0}{\xi_n}   +  \lambda  \bip{\Psi_n}{\tau_n^1}{v_n^1}{\xi_n}
 \\
 &
 \stackrel{(3)}{\leq}   (1-\lambda) \bpt{\tau_0 }{v_0}{\xi}  +\lambda \bpt{ \tau_1}{ v_1}{\xi},
 \end{aligned}
\]
where (1) follows from \eqref{def-p-gliminf}, (2) from the convexity of  the maps $\bip{\Psi_n}{\cdot}{\cdot}{\xi_n}$, and (3) from \eqref{hyp-p-limsup}.

With an analogous argument one proves that $\bpt {\cdot} {\cdot}\xi$ is $1$-positively homogeneous.
\end{proof}

We now show
 that, for $\tau>0$ the functional
$\bpt{\tau}{\cdot}{\cdot}$ has the same form  \eqref{def-bptreg}  as $\bip{\Psi_n} {\tau}{\cdot}{\cdot}$, cf.\ \eqref{structure-1-below}.
\begin{lemma}
\label{l:repr-psiri}
Assume Hypothesis \ref{hyp:p-lim}.
Let
$\psiri:
\R^d \to [0+\infty)$ be defined by
\begin{equation}
\label{psi-candidate}
\psiri(v) := \bpt 1v0.
\end{equation} Then,
$\psiri $ is a $1$-positively homogeneous dissipation potential,
   the sequence
$(\Psi_n)_n$ $\Gamma$-converges to $\psiri$, and thus
$(\Psi_n^*)_n$ $\Gamma$-converges to $\psiri^*$. Furthermore,
\begin{equation}
\label{structure-1-below}
\bpt \tau v \xi = \tau \psiri \left( \frac v\tau \right) + \tau \psiri^*(\xi) \quad \text{for all } \tau >0 \text{ and all }(v,\xi) \in
\R^d \times \R^d.
\end{equation}
\end{lemma}
\begin{proof}
Observe that $\psiri$   from \eqref{psi-candidate}   is convex and $1$-homogeneous thanks to Lemma \ref{l:3}.
It follows from \eqref{def-p-gliminf}
 and \eqref{hyp-p-limsup}, applied with the choices $\tau=1$ and $\xi=0$, that $\psiri = \gl \lim_{n \to \infty} \Psi_n$. Then, $(\Psi_n^*)_n$ $\Gamma$-converges to $\psiri^*$ by \cite[Thm.\ 2.18, p.\ 495]{Attouch}.
As a consequence  of
these convergences and of \eqref{def-bptreg}, we have 
 \eqref{structure-1-below}. 
\end{proof}

Our next two results address the characterization of $\bptname$ for $\tau=0$, providing a formula for  $\bptnew 0vw$
in the  two cases $\psiri^*(\xi)<+\infty$
 and $\psiri^*(\xi)=+\infty$.
 \begin{lemma}
 \label{l:zero-1}
 Assume Hypothesis \ref{hyp:p-lim}.
If $\psiri^*(\xi) < +\infty$, then
\begin{equation}
\label{1st-repr-psiri}
\bpt 0v \xi=\liminf_{\tau \to 0} \tau \psiri\left(\frac{v}{\tau}\right) = \psiri(v) \qquad \text{for all } v \in \R^d.
\end{equation}
\end{lemma}
\begin{proof}
It follows from \eqref{structure-1-below} and the fact that $\psiri^*(\xi) < +\infty$
that 
\begin{equation}
\label{leqlimf}
\bpt 0v \xi \leq \liminf_{\tau \to 0} \bpt \tau v \xi  \leq \liminf_{\tau \to 0} \tau \psiri\left(\frac{v}{\tau}\right).
\end{equation}

 To prove the converse inequality,
we
preliminarily
 observe that,  \RRR for any dissipation potential $\Psi$, 
for every $v \in \R^d$
 the map $\tau \mapsto \Psi \left( \tfrac{v}{\tau}\right) $ is non-increasing.
Therefore for all $0 <\tau<\sigma <1$ we have
\begin{equation}
\label{che-fatica}
\tau \Psi \left(\frac{v}{\tau}\right) \geq \sigma \Psi \left(\frac{v}{\sigma}\right). 
\end{equation} \EEE
Now,
let us fix a sequence $\xi_n \to \xi$ for which \eqref{hyp-p-limsup} holds, and accordingly
a sequence
$(\tau_n,v_n) \to (0,v)$ such that  $\bpt 0 v \xi = \liminf_{n \to \infty} ( \tau_n \Psi_n ({v_n}/{\tau_n}) + \tau_n \Psi_n^* (\xi_n) ) $.
It follows from  inequality \eqref{che-fatica} applied to  the functionals $\Psi_n$ that for every $\sigma \in (0,1)$
\[
\liminf_{n \to \infty} \left( \tau_n \Psi_n \left(\frac{v_n}{\tau_n}\right) + \tau_n \Psi_n^* (\xi_n) \right) \geq
\liminf_{n \to \infty} \left( \sigma \Psi_n \left(\frac{v_n}{\sigma}\right)  
\right)=
 \sigma \psiri \left(\frac{v}{\sigma}\right),
\]
where we have also exploited the positivity of the functionals  \RRR $\Psi_n^*$. \EEE
Therefore, in view of \eqref{hyp-p-limsup} we find
\[
\bpt 0v \xi \geq
 \sigma \psiri \left(\frac{v}{\sigma}\right) 
\]
and  conclude the converse of \eqref{leqlimf} passing to the limit as $\sigma \to 0$.
\end{proof}
\begin{lemma}
\label{l:2}
 Assume Hypothesis \ref{hyp:p-lim}.
If $\psiri^*(\xi) =+ \infty$, then
\begin{equation}
\label{2nd-repre}
\bpt 0v\xi=\gl \liminf_{n \to \infty} \inf_{\tau>0} \bip{\Psi_n}\tau v \xi \qquad \text{for all } v \in \R^d.
\end{equation}
\end{lemma}
\begin{proof}
Inequality $\geq$ follows from the  definition of $\bptname$. 
To prove the converse one,  
\RRR we may suppose that $v \neq 0$, since $\bpt 00{\xi} =0$. 
Take $(v_n,\xi_n)\to (v,\xi)$
 that attains $\gl \liminf_{n \to \infty} \inf_{\tau>0} \bip{\Psi_n}\tau v \xi$, i.e.\
$\inf_{\tau>0}  \bip{\Psi_n}\tau {v_n} {\xi_n}  \to \gl \liminf_{n
\to \infty} \inf_{\tau>0}  \bip{\Psi_n}\tau v \xi$. 
In particular,  $\liminf_{n\to\infty} \Psi_n^*(\xi_n)=+\infty$. Therefore, we may choose 
 $\bar{\tau}_n$ as
\[
\bar{\tau}_n \in \mathrm{Argmin}_{\tau>0} \left( \tau \Psi_n\left(\frac{v_n}{\tau}\right) + \tau \Psi_n^*(\xi_n) \right).
\]
Since $\liminf_{n\to\infty} \Psi_n^*(\xi_n)=+\infty$, \EEE it is clear that $\bar{\tau}_n \to 0$, hence
\[
\gl \liminf_{n \to \infty} \inf_{\tau>0}  \bip{\Psi_n} \tau v \xi =
\lim_{n \to \infty} \left( \bar{\tau}_n \Psi_n\left(\frac{v_n}{\bar{\tau}_n}\right) + \bar{\tau}_n \Psi_n^*(\xi_n) \right) \geq \bpt 0v{\xi}
\]
thanks to \eqref{def-p-gliminf}.
\end{proof}

We \RRR now prove  \EEE  a pseudo-monotonicity result for $\bptname$.
 \begin{lemma}
 \label{l:5}
 Assume Hypothesis \ref{hyp:p-lim}. Then,
 for every $\tau,\,\bar\tau \in [0,+\infty)$, $v,\,\bar{v} \in \R^d$ and  $\xi,\,\bar{\xi}\in \R^d$ we have that
\begin{equation}
\label{monotoniticity-p}
\bigg(\bpt \tau v \xi -  \bpt \tau v{\bar{\xi}} \bigg)\bigg(  \bpt {{\bar\tau}}{\bar{v}}\xi  - \bpt {\bar\tau}{\bar{v}}{\bar{\xi}}\bigg) \geq 0.
\end{equation}
\end{lemma}
\begin{proof}
Observe that \eqref{monotoniticity-p} holds  for the bipotentials $\bipname{\Psi_n}$: indeed, in that case it reduces to
 $ \tau \bar{\tau} (\Psi_n^*(\xi)-\Psi_n^*(\bar{\xi}))^2 \geq 0$.

Assume that $\bpt \tau v \xi > \bpt \tau v{\bar{\xi}} $ and choose $\bar{\xi}_n$ as in \eqref{hyp-p-limsup}
with $(\tau_n,v_n)$ such that
\begin{equation}
\label{from-limsup}
(\tau_n,v_n,\bar{\xi}_n) \to (\tau,v,\bar{\xi}), \qquad \bip{\Psi_n}{\tau_n}{v_n}{\bar{\xi}_n} \to \bpt{\tau}{v}{\bar{\xi}}.
\end{equation}
It follows from
 the definition
 \eqref{def-p-gliminf} of $\bptname$ that $\bpt{\tau}{v}{\xi} \leq \liminf_{n \to \infty} \bip{\Psi_n}{\tau_n}{v_n}{\xi_n}$
 for every sequence $\xi_n \to \xi$ in $\R^d$, and for $(\tau_n, v_n)$ as in \eqref{from-limsup}.
  Then
\begin{equation}
\label{ehsi}
0 < \bpt \tau v \xi - \bpt{\tau}{v}{\bar{\xi}} \leq \liminf_{n \to \infty}\bigg( \bip{\Psi_n}{\tau_n}{v_n}{\xi_n}- \bip{\Psi_n}{\tau_n}{v_n}{\bar{\xi}_n}\bigg).
\end{equation}
Therefore, for sufficiently big $n$ we have that
\begin{equation}
\label{at-level-n}
\bip{\Psi_n}{\tau_n}{v_n}{\xi_n}-\bip{\Psi_n}{\tau_n}{v_n}{\bar{\xi}_n} \geq 0.
\end{equation}

Now, again in view of \eqref{hyp-p-limsup}, choose
$\xi_n \to \xi$ (notice that \eqref{ehsi} holds for \emph{any} sequence $\xi_n$  converging to $\xi$) and
$\bar{\tau}_n\to \bar \tau$, $\bar{v}_n \to \bar v$ such that
$\limsup_{n \to \infty} 
 \bip{\Psi_n}  {{\bar\tau}_n}{\bar{v}_n}{\xi_n}  
 \leq \bpt {{\bar\tau}}{\bar{v}}\xi $. Since
$\liminf_{n \to \infty} \bip{\Psi_n}{{\bar\tau}_n}{\bar{v}_n}{\bar{\xi}_n} \geq \bpt {\bar\tau}{\bar{v}}{\bar{\xi}}$ by \eqref{def-p-gliminf}, we conclude that
\[
\bpt {{\bar\tau}}{\bar{v}}\xi  - \bpt {\bar\tau}{\bar{v}}{\bar{\xi}}
\geq \limsup_{n \to \infty} \left(  \bip{\Psi_n}  {{\bar\tau}_n}{\bar{v}_n}{\xi_n}  - \bip{\Psi_n} {{\bar\tau}_n}{\bar{v}_n}{\bar{\xi}_n}  \right) \geq 0,
\]
taking into account that
$
 \bip{\Psi_n}{{\bar\tau}_n}{\bar{v}_n}{\xi_n}  - \bip{\Psi_n}{{\bar\tau}_n}{\bar{v}_n}{\bar{\xi}_n}\geq 0
$ for sufficiently big $n$
thanks to \eqref{at-level-n} and the previously observed monotonicity property \eqref{monotoniticity-p} for
$\bipname{\Psi_n}$.  \RRR Thus, \eqref{monotoniticity-p} follows.  \EEE
\end{proof}
\par
Finally, let us consider  contact sets
associated with the bipotentials $\bipname{\Psi_n}$, i.e.\
\[
\Ctc{\bipname{\Psi_n}}:=  \left\{ (\tau,v,\xi) \in [0,+\infty) \times \R^d \times \R^d  \, : \langle v,\xi \rangle = \bip{\Psi_n} \tau v\xi  \right\}.
\]
Observe that for every $n\in \N$
\begin{enumerate}
\item  $\Ctc{\bipname{\Psi_n}} \cap \{0\} \times \R^d \times \R^d =  \{0\} \times \{ 0\} \times \R^d$;
\item for $\tau>0$, if $(\tau,v,\xi) \in \Ctc{\bipname{\Psi_n}}$, then
%
$\tau \in \argmin_{\sigma \in (0,+\infty)} ( \sigma\Psi_n (\tfrac
v{\sigma}) + \sigma\Psi_n^*(\xi))$.
\end{enumerate}
The following closedness property may be easily derived from 
\eqref{def-p-gliminf}.
\begin{lemma}
\label{l:closedness-contact-sets}
 Assume Hypothesis \ref{hyp:p-lim}. Then,
\begin{equation}
\label{closure-ctc}
\left\{
\begin{array}{ll}
  (\tau_n,v_n,\xi_n) \in \Ctc{\bipname{\Psi_n}},
\\
  (\tau_n,v_n,\xi_n) \to (\tau,v,\xi)
\end{array}
\right.
\quad \Rightarrow \quad (\tau,v,\xi) \in \Ctc{\bptname}.
\end{equation}
\end{lemma} 
\par
We are now in a position to carry out the
\underline{\textbf{proof of Theorem \ref{thm:generation}}} by verifying that $\bptname$ complies with properties   (1)--(5) from Definition \ref{def:visc-cont-pot}.  
\par
Properties   (1)\&(2) are guaranteed by Lemma \ref{l:3}, whereas (3) ensues from \eqref{structure-1-below}  in Lemma \ref{l:repr-psiri}. 
 Concerning property   (5),   
observe that \eqref{psiri-non-deg} 
ensues from  \eqref{structure-1-below} for $\tau>0$. For $\tau=0$,  it directly follows from \eqref{1st-repr-psiri} in the case $\psiri^*(\xi)<+\infty$, whereas for $\psiri^*(\xi)=+\infty$ we use
the monotonicity property
\eqref{monotoniticity-p}, giving
\[
(\bpt 1v{\xi} - \bpt 1v0) (\bpt 0v\xi - \bpt 0v0) \geq 0.
\]
Now,  $ \bpt 1v{\xi}  = \psiri(v) + \psiri^*(\xi) = +\infty$, hence we deduce that $ \bpt 0v\xi  \geq  \bpt 0v0 \geq \psiri(v)$ (here we have used that $\psiri^*(0) =0$).
\par   Under the additional \eqref{4-non-degeneracy},  it is immediate to check that $\psiri$ given by \eqref{psi-candidate} is non-degenerate,  i.e.\ property (4). 
This concludes the proof of Thm.\ \ref{thm:generation}. 
\QED
\section{\bf Main results}
\label{s:main} Let us consider a sequence $(\Psi_n)_n$ of
dissipation potentials on $\R^d$ with \emph{superlinear growth at
infinity}, namely   fulfilling \eqref{superlinear-infty} for every $n\in \N$. 
 It follows from
  \cite{mrs2013}, extending 
  the classical
results by \textsc{Colli\&Visintin} (cf.\
\cite{ColliVisintin90, Colli92}) that for every $n\in \N$ there
exists at least a solution $u \in \AC ([0,T];\R^d)$ of the Cauchy
problem for the generalized gradient system $(\Psi_n,\calE)$, with
$\calE$ complying with \eqref{smooth-energy}.  Namely,  $u$ solves the doubly nonlinear differential inclusion 
\begin{equation}
\label{CP-gen-gra}
\begin{cases}
\partial\Psi_n (\dot{u}(t)) + \rmD \ene t{u(t)} \ni 0 \qquad \foraa\, t\in (0,T),
\\
u(0) = u_0
\end{cases}
\end{equation}
for a given datum $u_0\in \R^d$. Furthermore, again an argument
(also often referred to as \emph{De Giorgi principle}, see
\cite{Mie-evolGamma}) based on the chain rule, 
 cf.\  \cite{mrs2013},  shows that a curve $u
\in \AC ([0,T];\R^d)$ is a solution of the gradient system
$(\Psi_n,\calE)$ if and only if it is a null-minimizer for the
(positive) functional
 of trajectories 
 $\mathcal{J}_{\Psi_n,\calE} : \AC ([0,T];\R^d)
\to [0,+\infty)$ defined by
\begin{equation}
\label{null-minim-gflow}
 \mathcal{J}_{\Psi_n,\calE} (u): = \int_0^T \left( \Psi_n (\dot u(s)) +  \Psi_n^* (-\rmD \ene s{u(s)})\right)   \dd s
+ \ene T{u_n(T)} - \ene 0{u_n(0)} - \int_0^T \partial_t \ene s{u_n(s)} \dd s.
\end{equation}
 Observe that  the positivity of 
$\mathcal{J}_{\Psi_n,\calE}  $ follows from 
\[
\begin{aligned}
\int_0^T \left( \Psi_n (\dot u(s)) +  \Psi_n^* (-\rmD \ene s{u(s)})\right) \dd s 
& \geq  -  \int_0^T  \langle \rmD \ene s{u(s)}, \dot u(s) \rangle \dd s 
\\
&
 = \ene 0{u_n(0)} - \ene T{u_n(T)} + \int_0^T \partial_t \ene s{u_n(s)} \dd s,
 \end{aligned}
\]
the last equality 
by the chain rule. \EEE 

The \textbf{main results of this paper, Theorems 
\ref{thm:liminf} and \ref{thm:limsupgenerale}} ahead, 
   concern the \textsc{Mosco}-convergence
to the functional $\Jfuname{\psiri,\bptname,\calE}$
 from \eqref{J-fun_RIS}, 
 with respect to the weak-strict topology of $\BV ([0,T];\R^d)$, of a family of functionals
  $(\Jfuname{\Psi_n,\calE})_n$ 
  suitably extending  $ \mathcal{J}_{\Psi_n,\calE} $ to  $ \BV ([0,T];\R^d)$.
Namely, we define
\begin{equation}
\label{extension-to-BV}
\Jfuname{\Psi_n,\calE} : \BV ([0,T];\R^d) \to [0,+\infty]  \quad \text{ by } \qquad
\Jfu{\Psi_n,\calE}{u} : = \begin{cases}
 \mathcal{J}_{\Psi_n,\calE} (u) & \text{ if } u \in \AC ([0,T];\R^d),
 \\
 +\infty & \text{ otherwise}.
\end{cases}
\end{equation}
\subsection{The $\Gamma$-liminf result}
\label{ss:5.1}
 First of all, let us fix the compactness properties of a sequence $(u_n)_n \subset \BV ([0,T];\R^d)$ with $\sup_{n} \Jfu{\Psi_n,\calE}{u_n} \leq C$,  assuming that the potentials  $\Psi_n$
comply with a suitable coercivity property.
\begin{proposition}
\label{prop:compactness}
Let $(\Psi_n)_n$ be a family of dissipation potentials with superlinear growth at
infinity and assume that
\begin{align}
\label{h:2}
\exists\, M_1,\, M_2>0 \ \ \forall\, n \in \N \ \ \forall\, \bv \in \R^d \, : \ \ \Psi_n(\bv) \geq M_1 \|\bv \|_1 -M_2.
\end{align}
Let $(u_n)_n \subset  \BV ([0,T];\R^d)$ fulfill
$ \|u_n(0)\| + \Jfu{\Psi_n,\calE}{u_n}  \leq C$ for some constant $C>0$  uniform w.r.t.\ $n\in \N$. Then,
there exist a subsequence $k\mapsto n_k$ and a curve $u$ such that $u_{n_k} \weakto u$ in $\BV ([0,T];\R^d)$.
\end{proposition}
We are now in a position to state the
$\gl \liminf$ result for the sequence $(\Jfuname{\Psi_n,\calE})_n$.
 Its proof is postponed to Section \ref{s:proofs}. 
\begin{theorem}
\label{thm:liminf} Let $(\Psi_n)_n$ be a family of dissipation
potentials with superlinear growth at infinity  such that the
associated bipotentials $(\bipname{\Psi_n})_n$ comply with
Hypothesis \ref{hyp:p-lim}, with limiting viscosity contact
potential   $\bptname$. Let $\psiri$ be the $1$-positively
homogeneous dissipation potential  defined by $\psiri (v) :=  \bpt
1v0$, \RRR and suppose that $\psiri$ is non-degenerate. \EEE

Then, for every  $(u_n)_n,\, u \in \BV ([0,T];\R^d)$ we have that 
\begin{equation}
\label{liminf-ineq}
u_n \weakto u \text{ in } \BV ([0,T];\R^d) \ \ \Rightarrow \ \ \liminf_{n\to \infty} \Jfu{\Psi_n, \calE}{u_n} \geq \Jfu{\psiri,\bptname,\calE}{u}.
\end{equation}
More precisely, we have as $n\to\infty$
\begin{align}
&
\label{energies+power}
\ene {t}{u_n(t)}\to \ene{t}{u(t)} \text{ and } \int_0^t \partial_t \ene r{u_n(r)}\dd r \to \int_0^t \partial_t  \ene r{u(r)} \dd r \quad \text{for every } t \in [0,T],
\\
&
\label{liminf_better}
\liminf_{n\to\infty} \int_s^t \left( \Psi_n(\dot{u}_n(r)) + \Psi_n^* (-\rmD \ene r{u_n(r)})\right) \dd r   \geq   \Var{\psiri,\bptname,\calE}{u}{s}t +  \int_s^t  \psiri^* (-\rmD \ene t{u(r)}) \dd r
\end{align}
 for every $0\leq s\leq t \leq T$. 
\end{theorem}
\noindent
\RRR Note that a sufficient condition for $\psiri$ to be non-degenerate
is that the potentials $\Psi_n$ comply with \eqref{4-non-degeneracy},
 cf.\ Theorem \ref{thm:generation}. 
A straightforward consequence  of Thm.\ \ref{thm:liminf} 
 is the following result. 
 \begin{corollary}
 \label{cor:liminf}
 Under the assumptions of Theorem \ref{thm:liminf}, let $(u_n)_n \subset \AC ([0,T];\R^d)$  \RRR fulfill $\Jfu{\Psi_n}{u_n} \leq \eps_n $ 
 for every $n\in \N$, 
 for some vanishing sequence $(\eps_n)_n$.  \EEE

 Then, any limit point $u$ of $(u_n)_n$ with respect to the weak-$\BV([0,T];\R^d)$-topology is a Balanced Viscosity solution to the rate-independent
 system $\RIS$,  and, up to a subsequence,  convergences \eqref{energies+power} and
  \begin{equation}
  \label{lim-better}
  \lim_{n\to\infty} \int_s^t \left( \Psi_n(\dot{u}_n(r)) + \Psi_n^* (-\rmD \ene r{u_n(r)})\right) \dd r   =   \Var{\psiri,\bptname,\calE}{u}{s}t +  \int_s^t  \psiri^* (-\rmD \ene t{u(r)}) \dd r
  \end{equation}
 hold  for all $0\leq s \leq t \leq T$.
 \end{corollary}

\subsection{Examples}
\label{s:example}
\RRR We now focus  on two classes of dissipations potentials $(\Psi_n)_n$, with superlinear growth at infinity, approximating  
 a $1$-positively homogeneous dissipation potential $\psiri$.  In the first case, the dissipation potentials $\Psi_n$ are obtained by rescaling from a given dissipation potential $\Psi$
with superlinear growth at infinity, and suitably converge to some $\psiri$. In the  second case, we consider the stochastic model introduced in Section 
\ref{s:2} and the associated potentials $\Psi_n$ given  by \eqref{def:cosh_psi}:
the limiting potential is $\psiri(v) = A \|v\|_1$.   We will show that, in both cases Hypothesis \ref{hyp:p-lim} is fulfilled. 
\subsubsection*{\bf The vanishing-viscosity approximation}
We consider the dissipation potentials  
\begin{subequations}
\label{van-visc-case}
\begin{equation}
\label{van-visc-case-2}
\Psi_n(v) = \Psi_{\eps_n}(v) = \frac1{\eps_n} \Psi(\eps_n v) \quad \text{for all } v \in \R^d, \text{ with $\eps_n\down 0 $},
\end{equation}
with $\Psi :\R^d \to [0,+\infty) $ a fixed potential with superlinear growth at infinity.   \EEE
We suppose that  there exists a $1$-homogeneous dissipation potential $\psiri$ such that
\begin{equation}
\label{van-visc-case-3}
\psiri(v) = \lim_{n\to\infty}\Psi_n(v) = \lim_{n\to\infty} \frac1{\eps_n} \Psi(\eps_n v)\quad \text{for all } v \in \R^d.
\end{equation}
\end{subequations}
\EEE
\begin{example}
\label{ex:van-visc}
\upshape
In particular, we focus on these two cases (cf.\ \cite[Ex.\ 2.3]{MRS10}):
\begin{enumerate}
\item \textbf{$\psiri$-viscosity:} the superlinear dissipation potential $\Psi$ is obtained augmenting $\psiri$
by a superlinear function of $\psiri$ itself. Namely,
 given   a convex superlinear function  $F_V: [0,+\infty) \to [0,+\infty)$, we set
\begin{equation}
\label{form-psi-visc-single}
\Psi(v) : = \psiri(v) + F_V(\psiri(v)), \quad \text{whence } \Psi_n(v) = \psiri(v) + \frac1{\eps_n} F_V(\eps_n \psiri(v)) \quad \text{for all } v \in \R^d.
\end{equation}
To fix ideas, we may think of $\psiri(v) = A \| v\|_1$ and $F_V(\rho) = \frac12\rho^2$, giving rise to
\begin{equation}
\label{specific-psi-visc}
\Psi_n(v) = A \| v\|_1 +  \frac{\eps_n}{2} A^2  \| v\|_1^2.
\end{equation}
\item \textbf{$2$-norm vanishing-viscosity:}  Let us now consider a norm $\|\cdot\|$ on $\R^d$, 
 different from that associated with $\psiri$. 
We set
\begin{equation}
\label{form-psi-visc} \Psi(v) : = \psiri(v) + F_V( \|v\|), \quad
\text{whence } \Psi_n(v) = \psiri(v) + \frac1{\eps_n}  F_V(\eps_n
\|v\|)\quad \text{for all } v \in \R^d,
\end{equation}
with again $F_V: [0,+\infty) \to [0,+\infty)$ convex and
superlinear.  In this way we generate, for example, the dissipation
potentials
\begin{equation}
\label{specific-psi-visc-2} \Psi_n(v) = A \| v\|_1 + \frac{\eps_n}{2}
\| v\|_2^2,
\end{equation}
with  $ \| \bv \|_2 : = \left(
\sum_{i=1}^d |v_i|^2\right)^{1/2}$. 
\end{enumerate}
\end{example}
This family of  dissipation potentials
 comply with the hypotheses of Thm.\ \ref{thm:liminf}.
\begin{proposition}
\label{l:van-visc-ok-hyp}
The dissipation potentials  from \eqref{van-visc-case} comply with
\RRR \eqref{4-non-degeneracy} \EEE and with 
 Hypothesis \ref{hyp:p-lim}, where
\begin{equation}
\label{p-lim-van-visc}
\bptname: [0,+\infty) \times \R^d \times \R^d \to [0,+\infty] \qquad \text{is given by }
\bpt\tau v\xi:=
\begin{cases}
\psiri(v) + I_{K^*}(\xi) & \text{if } \tau>0,
\\
\inf_{\eps_n>0}\left(\Psi_{\eps_n}(v) + \Psi_{\eps_n}^*(\xi)  \right) & \text{if } \tau=0.
\end{cases}
\end{equation}
\end{proposition}
The \emph{proof} can be straightforwardly retrieved from the argument for \cite[Lemma 6.1]{MRS10}.
\begin{example}[Example \ref{ex:van-visc}  continued]
\upshape Following \cite[Rem.\ 3.1]{MRS10}, we explicitly calculate
$\bpt0{v}{\xi}$, 
using formula \eqref{p-lim-van-visc}, in the two cases
of Example \ref{ex:van-visc}:
\begin{enumerate}
\item \textbf{$\psiri$-viscosity:}
We have
\[
\bpt 0v{\xi} : = \begin{cases} \psiri(v) & \text{if } \xi \in K^*,
\\
\psiri(v) \sup_{v\neq 0} \frac{\langle \xi,v \rangle}{\psiri(v)} &
\text{if } \xi \notin K^*.
\end{cases}
\]
Therefore, in the particular case
 $\psiri(v) = A \| v\|_1$, taking into account that
\[
K^* = \overline{B}_A^{\infty} (0) := \{ \xi \in \R^d\, : \
\|\xi\|_\infty \leq A \},
\]
 we retrieve the formula
\begin{equation}
\label{yes} \bpt 0v{\xi} = \|v\|_1 (A \vee \|\xi\|_\infty)
\end{equation}
\RRR (here and in what follows, we use the notation $a\vee b$ for $\max\{a,b\}$). \EEE
\item \textbf{$2$-norm vanishing-viscosity:}  In this case, we have
\begin{equation}
\label{bpt-0-visc}
 \bpt 0v{\xi} = \psiri(v) + \|v\|\min_{\zeta\in K^*} \|
 \xi-\zeta\|_*,
\end{equation}
where we have used the notation $\| \zeta\|_* : = \sup_{v\neq 0}
\tfrac{\langle \zeta, v \rangle}{\|v\|}$. \RRR Clearly, 
\eqref{yes} is a particular case of \eqref{bpt-0-visc}. \EEE
\end{enumerate}
\end{example}
\subsubsection*{\bf The stochastic approximation}
We now consider the dissipation potentials $\Psi_n$ from \eqref{def:cosh_psi}, i.e.\
\begin{equation}
\label{Psi-n-bis}
\begin{gathered}
\Psi_n(\bv) = \sum_{i=1}^d \psi_n(v_i) =  \sum_{i=1}^d \frac{v_i}{n} \log \left( \frac{ v_i + \sqrt{v_i^2 + e^{-2n A}}}{e^{-n A}} \right) - \frac{1}{n}\sqrt{v_i^2 + e^{-2n A}} + \frac{e^{-n A}}{n},
\\
\text{with }
\Psi^*_n(\xi) = \sum_{i=1}^d \psi^*_n (\xi_i) = \sum_{i=1}^d \frac{e^{-n A}}{n} \left( \cosh ( n \xi_i ) - 1 \right).
\end{gathered}
\end{equation}
Preliminarily, we observe that 
\begin{equation}
\label{pointwise&gamma}
\begin{cases}
\Psi_n(v) \to \Psi_0(v) = A \| v \|_1  \quad \text{for all } v \in \R^d, \text{ and } \gl \lim_{n\to\infty} \Psi_n = \psiri,
\\
 \Psi_n^* (\xi) \to 
  I_{K^*}(\xi) \text{ with } K^*=\overline{B}_A^{\infty} (0) \quad \text{for all } \xi \in \R^d, \text{ and } \gl \lim_{n\to\infty} \Psi_n^* = \psiri^*.
 \end{cases}
\end{equation}
In order to check the above statement, e.g.\ for $\Psi_n(v)$, it is sufficient to recall that $\Psi_n(v) = \sum_{i=1}^d \psi_n(v_i)$, and that the  real  functions $(\psi_n)_n$
pointwise  and $\Gamma$-converge to the $1$-positively homogeneous potential $\psi_0:\R \to \R$ given by  $\psi_0(v) = A |v|$. 
We will now prove that the counterpart to \RRR Proposition \EEE
\ref{l:van-visc-ok-hyp} holds.
\begin{proposition}
\label{l:large-deviations} The dissipation potentials from 
 \eqref{Psi-n-bis} comply \RRR with \eqref{4-non-degeneracy} and with \EEE Hypothesis \ref{hyp:p-lim}, with  \RRR  limiting viscosity contact potential \EEE
\begin{equation}
\label{p-lim-van-visc-bis} \bptname: [0,+\infty) \times \R^d \times \R^d
\to [0,+\infty]   \text{ given by } \bpt\tau v\xi:=
\begin{cases}
\psiri(v) + I_{K^*}(\xi) & \text{if } \tau>0,
\\
\|v\|_1 (A \vee \|\xi\|_\infty) & \text{if } \tau=0.
\end{cases}
\end{equation}
\end{proposition}
\begin{proof}
We will split the proof in several claims.
\medskip

\noindent 
\textbf{Claim $1$:  \RRR \eqref{p-lim-van-visc-bis} \EEE  holds for $\tau>0$.}
It follows from the $\Gamma$-convergence \RRR properties in 
\eqref{pointwise&gamma} \EEE that
 $\bptname = \gl \liminf_{n\to\infty} \bipname{\Psi_n}$ fulfills  
$\bpt{\tau}{v}{\xi} \geq \psiri(v) + I_{K^*}(\xi) $  for all $(v,\xi)\in \R^d \times \R^d$,  if  $\tau>0$. For the converse inequality, 
for every $\xi\in \R^d$ we take the constant recovery sequence $\xi_n \equiv \xi$ and  again choose for fixed $(\tau,v) \in [0,+\infty) \times \R^d$ 
the sequences $\tau_n \equiv \tau$ and $v_n \equiv v$. The pointwise convergences from \eqref{pointwise&gamma} ensure that 
\[
\bpt{\tau}{v}{\xi} \leq \limsup_{n\to\infty} \bip{\Psi_n}{\tau}{v}{\xi} =\tau \psiri\left(\frac v\tau \right) +\tau I_{K^*}(\xi)  = \psiri(v) + I_{K^*}(\xi). 
\]
 Hence we conclude that $\bpt{\tau}{v}{\xi}=\psiri(v) + I_{K^*}(\xi)$, i.e.\
\eqref{p-lim-van-visc} for $\tau>0$. 
\medskip

\noindent 
\textbf{Claim $2$:  \RRR \eqref{p-lim-van-visc-bis} \EEE  holds for $\tau=0$ and $v=0$.} In this case we have to check that $\bpt 00{\xi} =0$, which is equivalent to showing that 
$\bpt 00{\xi} \leq 0$ as  the functional $\bptname$ is  positive. To this  aim, for every  fixed $\xi\in\R^d$ we observe that for any null sequence $\tau_n\downarrow 0$
\[
\bpt 00{\xi} \leq \limsup_{n\to\infty} \bip{\Psi_n}{\tau_n}{0}{\xi} = \limsup_{n\to\infty}\tau_n  \RRR \Psi_n^*(\xi), \EEE
\]
and then we choose $(\tau_n)_n$ vanishing fast enough in such a way that the $\limsup$ on the right-hand side equals zero.

\noindent 
\textbf{Claim $3$: \RRR \eqref{p-lim-van-visc-bis} \EEE holds for $\tau=0$ and $v\neq0$.}
We will split the proof  in several (sub-)claims. In the following calculations, taking into account that $\Psi_n = \sum_{i=1}^d \psi_n$ and 
 $\Psi_n^* = \sum_{i=1}^d \psi_n^*$ with $\psi_n:\R \to \R$ and $\psi_n^* : \R \to \R$
even functions, we will often confine the discussion to 
the case in which $v = (v_1, \ldots, v_d)$ fulfills
 $v_i \geq 0$ for all $i=1, \ldots, d$, and analogously for \RRR  $\xi = (\xi_1, \ldots, \xi_d) $. \EEE

 Moreover, 
 we will need to work with 
the
modified bipotentials  \RRR $\bipnamed{\Psi_n}{\delta} : [0,+\infty) \times \R^d \times \R^d \to [0,+\infty]$ given by \EEE
 \begin{equation}
\label{bipd}
\bipd{\Psi_n}{\delta}\tau{v}\xi := 
\begin{cases}
 \tau \Psi_n \left( \frac v\tau \right) + \tau \Psi_n^*(\xi)  + \tau \delta &  \text{for } \tau>0,
 \\
 0   & \text{for } \tau=0, \ v=0,
 \\
 +\infty  & \text{for } \tau =0 \text{ and } v \neq 0
 \end{cases}
\end{equation}
with $\delta>0$ \underline{fixed}. 
We remark that 
$\argmin_{\tau>0} \bipd{\Psi_n}{\delta}\tau{v}\xi \neq \emptyset$. Since for every $\xi \in \R^d$ the map $v\mapsto 
\min_{\tau>0} \bipd{\Psi_n}{\delta}\tau{v}\xi $ is  in turn $1$-positively homogeneous, there exists a closed convex set $ K_{n,\delta}^*(\xi) $ such that 
\begin{subequations}
\label{supp-re}
\begin{equation}
\label{support-repre}
\begin{gathered}
\min_{\tau>0} \bipname{\Psi_n}^\delta(\tau,v,\xi) = \sup \left\{ \langle v,  w \rangle  \ : \ w \in K_{n,\delta}^*(\xi) \right\}.
\end{gathered}
\end{equation}
Indeed, it turns out  \RRR (cf.\ 
\cite[Thm.\ A.17]{MRS10}) \EEE
that 
  \begin{equation}
\label{support-repre-1}
\begin{gathered}
 K_{n,\delta}^*(\xi)  =  \{ w \in \R^d \, : \ \Psi_n^*(w) \leq\Psi_n^*(\xi) + \delta\}.
 \end{gathered}
\end{equation} 
\end{subequations}
We need an intermediate estimate before proving the $\geq$-inequality in \eqref{p-lim-van-visc-bis}, \RRR i.e.\ \eqref{lower-bound} below. \EEE

\noindent 
\textbf{Claim $3.1$:} 
there holds
\begin{equation}
\label{lower-bound-used}
\bptname (0,\bv,\xi)
 \geq \inf\{ \liminf_{n \to \infty} \bipname{\Psi_n}^\delta(\bar \tau_n^\delta ,v_n,\xi_n)\, : v_n \to v,  \ \xi_n \to \xi  \},  \qquad \text{where }
 \bar \tau_n^\delta \in \argmin_{\tau}  \bipname{\Psi_n}^\delta(\tau ,v_n,\xi_n)\,.
 \end{equation}
 This follows from 
\[
\begin{split}
\bptname (0,\bv,\xi) & = \inf\{ \liminf_{n \to \infty} \bip{\Psi_n}{\tau_n}{v_n}{\xi_n}\, : \tau_n \to 0, \ v_n \to v,  \ \xi_n \to \xi  \} \\
& =  \inf\{ \liminf_{n \to \infty} \left( \bipname{\Psi_n}^\delta(\tau_n,v_n,\xi_n)\right)  - \delta \tau_n\, : \tau_n \to 0, \ v_n \to v,  \ \xi_n \to \xi  \} \\
& \stackrel{(2)}{\geq} \inf\{ \liminf_{n \to \infty}   \min_{\tau>0}\bipname{\Psi_n}^\delta(\tau ,v_n,\xi_n) \, : v_n \to v,  \ \xi_n \to \xi  \}\,,
\end{split}
\]
where (2) follows from the fact that $\lim_{n\to\infty}\delta\tau_n=0$ for every vanishing sequence $(\tau_n)$\,.
\smallskip

\noindent 
\textbf{Claim $3.2$:}  there holds
\begin{equation}
\label{lower-bound}
\bptname (0,v,\xi)  \geq \| v \|_1 (A \vee \|\xi \|\f).
\end{equation}
In view of  \eqref{lower-bound-used}, it is sufficient to prove that 
\begin{equation}
\label{ad-low-bd}
\inf\{ \liminf_{n \to \infty} \bipname{\Psi_n}^\delta(\bar \tau_n^\delta ,v_n,\xi_n)\, : v_n \to v,  \ \xi_n \to \xi  \} \geq \| v \|_1 (A \vee \|\xi \|\f).
\end{equation}
Hence, we
fix a sequence $(v_n,\xi_n) \to (v,\xi)$ and,  for \RRR $n$ sufficiently big such that $\frac1n \log d < A$, \EEE define
$w_n \in \R^d$ by
\[
w_n:=\left( \left(A \vee \|\xi_n\|\f  \right)- \frac1n \log d \, , \, \cdots \, , \, \left(A \vee \|\xi_n\|\f  \right) - \frac1n \log d \right)\,.
\]
Taking into account 
the form \eqref{def:cosh_psi*} of $\Psi_n^*$, we estimate
\[
\Psi_n^*(w_n)  =  d \frac{ e^{-nA}}{n}\left( \cosh ( n\| w_n \|\f) -1 \right)  
\]
 distinguishing the two cases $ \|\xi_n\|\f  \leq A$ and $ \|\xi_n\|\f  > A$.
In the former situation,
it is sufficient to observe that 
$
\| w_n \|\f \leq A,
$
so that 
\begin{subequations}
\label{myst-Gio}
\begin{equation}
\label{myst-Gio-1}
\Psi_n^*(w_n)  \leq d \frac{ e^{-nA}}{n} \left( \cosh(nA) -1\right) =\frac dn \left( \frac{1+ e^{-2nA} - 2e^{-nA}}{2} \right)\leq \frac d n \leq \delta
\end{equation}
for $n$ sufficiently big. In the case  $ \|\xi_n\|\f  > A$, 
we use that 
\begin{equation}
\label{myst-Gio-2}
\begin{aligned}
\Psi_n^*(w_n) 
 & = d \frac{ e^{-nA}}{n} \left( \cosh(n \|\xi_n\|\f   - \log(d)) -1 \right) 
  \\ & =  d \frac{ e^{-nA}}{2n}   \left( e^{n\| \xi_n \|\f - \log d} +e^{-n \| \xi_n \|\f + \log d} - 2 \right) 
\\ & =   \frac{ e^{-nA}}{2n}   \left( e^{n\| \xi_n \|\f } +d^2e^{-n \| \xi_n \|\f} - 2d \right) 
\\ & =   \frac{ e^{-nA}}{2n}   \left( e^{n\| \xi_n \|\f } +e^{-n \| \xi_n \|\f} - 2d \right)  + 
 \frac{ e^{-nA}}{2n}  (d^2-1) e^{-n \| \xi_n \|\f} 
 \leq   \Psi_n^*(\xi_n)   +\delta
\end{aligned}
\end{equation}
for $n$ sufficiently big such that $\frac{d^2-1}{2n } \leq \delta$. 
\end{subequations}
All in all, \eqref{myst-Gio} gives that 
\[
\Psi_n^*(w_n)  \leq  \Psi_n^*(\xi_n)   +\delta,
\]
which implies that $w_n \in K_{n,\delta}(\xi_n)$,  for all $n $ sufficiently big. 
Now,  
using the representation formula \eqref{support-repre} for $\bipname{\Psi_n}^{\delta}(\bar \tau_n^\delta,\cdot,\cdot)$, we find 
\[
 \bipname{\Psi_n}^\delta (\bar \tau_n^\delta,v_n,\xi_n) \geq \, \langle v_n,  w_n \rangle  =  \| v_n \|_1 (A \vee \|\xi_n\|\f) - \frac1n \log d \| v_n \|_1,
\]
where the last equality follows from the fact that $v_n = (v_n^1,\ldots, v_n^d)$ fulfills $v_n^i
\geq 0$ for all $i=1,\ldots, d$. 
Hence 
$\liminf_{n\to\infty} \bipname{\Psi_n}^\delta (\bar \tau_n^\delta,v_n,\xi_n) \geq \| v \|_1 (A \vee \|\xi \|\f)  $ and, since the sequences $(v_n)_n$ and $(\xi_n)_n$ are arbitrary, we conclude \eqref{ad-low-bd},   and thus \eqref{lower-bound}. 
\par 
In order to prove the converse \RRR of  inequality 
\eqref{lower-bound}, \EEE
and conclude 
\eqref{p-lim-van-visc-bis}, we 
  preliminarily  need to investigate the properties of the sets $K_{n,\delta}^*$. 
\smallskip

\noindent 
\textbf{Claim $3.3$:} there holds
 \begin{equation}
 \label{eq:bound_xin}
\forall\, \delta>0 \ \ \exists\, n_\delta \in \N \  \forall\, n \geq n_\delta \ \forall\,  w \in  K_{n,\delta}^*(\xi): \quad   \| w \|\f \leq A\vee \| \xi \|\f 
+ \frac1n \log(2e n \delta)\,.
 \end{equation}
 Indeed, every $w\in  K_{n,\delta}^*(\xi)$ fulfills 
$\Psi_n^*(w) \leq \Psi_n^*(\xi) + \delta$. Using  the explicit formula
for $\Psi_n^*$
 we obtain that
 \[
\frac{e^{-nA}}{n}\cosh (n\| w \|\f ) \leq \frac{d e^{-nA}}{n}\cosh (n\| \xi \|\f ) + \delta,
 \]
 whereby
 \[
 \frac{e^{-nA}}{2n}  e^{n \|w \|\f} \leq  
  \frac{de^{-nA}}{2n} 
  e^{n \| \xi \|\f }  +   \frac{de^{-nA}}{2n}  +\delta \leq 
 \frac{de^{-nA}}{n} 
  e^{n \| \xi \|\f } + \delta,
 \]
 and thus
 \[
  \| w \|\f   \leq \frac1n \log \left( 2n\delta e^{nA} + 2 d e^{n \| \xi \|\f }\right).
 \]
 Now, doing some algebraic manipulation on the logarithmic term on the right-hand side \RRR we find \EEE
 \[
 \begin{split}
\log \left( 2n\delta e^{nA} + 2 d e^{n \| \xi \|\f }\right) & = \log \left(  e^{nA + \log n\delta} \left(1 +  e^{n (\| \xi \|\f - A) + \log d - \log n\delta}\right) \right) + \log 2 \\
& \stackrel{(1)}{\leq} \log \left( 1 +  e^{n (\| \xi \|\f - A)_+}\right) + nA + \log n\delta + \log 2 \\
& \stackrel{(2)}{\leq}   n (A  \vee \| \xi \|\f )+ 1+ \log 2n\delta,
 \end{split}
 \]
 where for (1) we have used that $n\delta > d$ for $n$ sufficiently big
  and for (2) we have estimated
 $\log( 1 +  e^{n (\| \xi \|\f - A)_+})  = \log (e^{n (\| \xi \|\f - A)_+})+ 
  \log (e^{-n (\| \xi \|\f - A)_+}+1) \leq  \log (e^{n (\| \xi \|\f - A)_+}) +1 $.  
  Then,  \eqref{eq:bound_xin} ensues. 
 \smallskip

\noindent 
\textbf{Claim $3.4$:}  for every $(v,\xi) \in \R^d \times \R^d$ and all sequences $(v_n)_n,\, (\xi_n)_n$ with $v_n \to v $ and $\xi_n \to \xi$, for every 
 $\bar \tau_n^\delta \in \argmin \bipname{\Psi_n}^\delta (\cdot,v_n,\xi_n)$ there holds
 \begin{equation}
 \label{lim-tau-nd}
 \lim_{n\to \infty} \bar \tau_n^\delta =0.
 \end{equation}
 We distinguish two cases:
(1)  $\psiri^*(\xi) =+\infty$ and (2)  $\psiri^*(\xi) =0$. 
\begin{enumerate}
\item
In the first case, 
  we have  $\liminf_{n\to\infty} \Psi_n^*(\xi_n) =+\infty$. Then  $\bar \tau_n^\delta$ must be vanishing to ``cancel'' the $\tau \Psi_n^*$-contribution, cf.\ also the proof of Lemma \ref{l:2}. 
  \item
\RRR  In the second case,  to show  \eqref{lim-tau-nd} we will provide  an estimate from above for $ \bar \tau_n^\delta$  by exploiting 
the Euler-Lagrange equation for the minimum problem $
\min_{\tau>0} \bipname{\Psi_n}^\delta(\tau,v,\xi) $. Namely, 
$ \bar \tau_n^\delta$  complies with 
\begin{equation}
\label{Euler-Lagrange-tau}
0 \in 
\partial_\tau \bipname{\Psi_n}^\delta(\cdot,v_n,\xi_n)(\bar \tau_n^\delta) = \Psi_n\left( \frac{v_n}{\bar \tau_n^\delta} \right) - 
\left\langle \partial \Psi_n \left( \frac{v_n}{\bar \tau_n^\delta} \right), 
\frac{v_n}{\bar \tau_n^\delta} \right \rangle  + \Psi_n^*(\xi_n) + \delta
\end{equation}
(where, with a slight abuse of notation, we have written $ \partial \Psi_n ( v_n/\bar \tau_n^\delta )$ as a  singleton). 
\EEE
Using the explicit formula \eqref{Psi-n-bis} for $\Psi_n$ we find
\[
\Psi_n\left( \frac{v_n}{\bar \tau_n^\delta} \right) -\left\langle \partial \Psi_n \left( \frac{v_n}{\bar \tau_n^\delta} \right), 
\frac{v_n}{\bar \tau_n^\delta} \right \rangle = \frac{de^{-nA}}{n} - \sum_i \frac1n \sqrt{\frac{\RRR (v_{n}^i)^2 \EEE}{ (\bar \tau_n^\delta)^2} + e^{-2nA}}\,.
\]
\RRR Therefore, \eqref{Euler-Lagrange-tau} yields \EEE
\[
n\delta + de^{-nA} + n\Psi_n^*(\xi_n)=  \sum_i \sqrt{\frac{\RRR (v_{n}^i)^2 \EEE}{(\bar \tau_n^\delta)^2} + e^{-2nA}} \leq d \sqrt{\frac{\| v_n \|\f^2}{(\bar\tau_n^\delta)^2} + e^{-2nA}}, 
\]
\begin{equation}
\label{estimate_tau_n_delta}
\text{whence}  \quad  (\bar  \tau_n^\delta)^2 \leq \frac{d^2 \| v_n \|\f^2}{n^2 \delta^2 + n^2 \left(\Psi_n^*(\xi_n)\right)^2} \to 0 \quad \text{ for } n \to \infty.
\end{equation}
\end{enumerate}
 We are now in a position to conclude the proof of \eqref{p-lim-van-visc-bis}.
 
\noindent 
\textbf{Claim $3.5$:}  there holds
\begin{equation}
\label{lower-bound-this}
\bptname (0,v,\xi)  \leq \| v \|_1 (A \vee \|\xi \|\f).
\end{equation}
We will in fact prove that 
 \begin{equation}
 \label{gamma-limsup-cond-for-p}
 \forall\, \xi \in \R^d \ \ \exists\, (\xi_n)_n \subset \R^d\, :  
 \ \ \forall\, v \in   \R^d \ \exists\, (\tau_n,v_n)_n\, \text{ s.t. } 
 \begin{cases}
 \tau_n\to 0 , \\ v_n \to v, \\  \limsup_{n \to \infty} \bip {\Psi_n}{\tau_n}{v_n}{\xi_n} \leq \| v \|_1 (A \vee \|\xi \|\f).
 \end{cases}
 \end{equation}
 Taking into account that
  $\bptname = \gl \liminf_{n\to\infty} \bipname{\Psi_n}$, we will then conclude \eqref{lower-bound-this}. 
 To check \eqref{gamma-limsup-cond-for-p},  let us choose the constant recovery sequences $\xi_n \equiv \xi$ and $v_n \equiv v$, 
 and let  $\tau_n: =  \bar \tau_n^\delta \in \argmin_{\tau>0} \bipd{\Psi_n}{\delta}{\tau}{v}{\xi}$. By the previous Claim 3.4, we have that $\tau_n \down 0$. 
 Now,  in view of  the representation formula \eqref{support-repre} for $\min_{\tau>0} \bipname{\Psi_n}^\delta(\tau,v,\xi)$, we can construct a sequence 
 $\{ \tilde{\bxi}_{n} \} \subset K_{n,\delta}^*(\xi)$ such that  
 \[
 \bipd{\Psi_n}{\delta}{\bar \tau_n^\delta}{v}{\xi} \leq \langle v, \tilde{\bxi}_{n} \rangle + \frac1n \leq   \| \bv \|_1
(A \vee \| \xi \|\f) +  \frac{\| \bv \|_1}n \log(2 e n \delta)
 + \frac1n,
 \]
 where the second estimate ensues from \eqref{eq:bound_xin}.
 Therefore $ \limsup_{n\to\infty}\bipd{\Psi_n}{\delta}{\bar \tau_n^\delta}{v}{\xi} \leq  \| v \|_1 (A \vee \|\xi \|\f) $. Since
 $\limsup_{n\to\infty}\bipd{\Psi_n}{\delta}{\bar \tau_n^\delta}{v}{\xi} = \limsup_{n\to\infty}\bip{\Psi_n}{\bar \tau_n^\delta}{v}{\xi}$ as the sequence $(\bar \tau_n^\delta)_n$ is vanishing, we conclude \eqref{gamma-limsup-cond-for-p}.
 
 This finishes the proof of Proposition \ref{l:large-deviations}. 
\end{proof}
\subsection{The $\Gamma$-limsup result}
\label{ss:5.3}
\RRR For the $\Gamma$-$\limsup$ counterpart to Theorem \ref{thm:liminf}, 
where we now consider the \emph{strict} topology in $\BV ([0,T];\R^d)$,
we will focus on 
the $1$-positively homogeneous potential
\[
\psiri(v) = A \| v \|_1 \quad \text{with } A>0,
\]
and the 
 two following  \emph{specific cases}: 
\begin{itemize}
\item[\textbf{vanishing viscosity:}]  the dissipation potentials $\Psi_n$  are obtained by augmenting $\psiri$ by a quadratic term involving a (possibly) different norm $\| \cdot \|$ (cf.\ \ref{form-psi-visc}),
i.e.
\begin{equation}
\label{specific-1-limsup}
\begin{gathered}
 \Psi_n(v) = A \| v\|_1 + \frac{\eps_n}{2}
\| v\|^2 \text{ with } \eps_n \down 0,  \\
\text{with limiting viscosity contact potential }
  \bpt \tau{v}{\xi} =
  \begin{cases}
  \psiri(v) + I_{K^*}(\xi) & \text{if } \tau>0,
  \\
   \psiri(v) + \|v\|\min_{\zeta\in K^*} \|
 \xi-\zeta\|_* & \text{if } \tau=0; 
 \end{cases}
 \end{gathered}
\end{equation}
\item[\textbf{stochastic approximation:}]   the dissipation potentials $\Psi_n$ are given by \eqref{Psi-n-bis}, with  viscosity contact potential 
\[
  \bpt \tau v{\xi} = 
\begin{cases}
\psiri(v) + I_{K^*}(\xi)  &    \text{ if $\tau>0$,} 
  \\
 \|v\|_1 (A \vee \|\xi\|_\infty)  &    \text{ if $\tau=0$.} 
 \end{cases}
 \] 
\end{itemize}
\par
In  \cite{BonaschiPeletier14}, which focused on \emph{one-dimensional} rate-independent systems, 
 the $\Gamma$-$\limsup$ result was obtained in a much larger generality, for a class of dissipation potentials $\Psi_n$
fulfilling suitable growth conditions and other properties. Such properties are satisfied in the two abovementioned particular cases. 
\par
We believe that, to some extent, the results in \cite{BonaschiPeletier14} could be extended to the
present
 multi-dimensional context.
Still, we have preferred to confine the discussion to the vanishing-viscosity and the stochastic approximations, in order to develop more explicit calculations than those in  the proof of 
\cite[Thm.\ 4.2]{BonaschiPeletier14},
significantly exploiting the specific structure of these examples. 
\par
Finally,  let us mention in advance that, like in \cite{BonaschiPeletier14},  we will need to impose some enhanced regularity for $\calE (t,\cdot)$, namely
\begin{equation}
\label{enhanced-E}
\begin{gathered}
\! \!\!\!
\exists\, C_{\mathsf{E}}>0 \ \ \forall\, (t,u) \in [0,T]\times \R^d\, : 
\| \rmD \ene tu \| \leq  C_{\mathsf{E}} \quad \text{and } \rmD \ene{\cdot}{u} \text{ is uniformly Lipschitz continuous, i.e. } 
\\
 \exists\,   L_{\mathsf{E}}>0  \ \forall\, t_1,\, t_2 \in [0,T] \ \forall\, u \in \R^d \, : \quad \| \rmD \ene{t_1}u -  \rmD \ene{t_2}u \| \leq  L_{\mathsf{E}}|t_1 -t_2|\,.
\end{gathered}
\end{equation}
\begin{theorem}
\label{thm:limsupgenerale}
Let $\calE$ comply with \eqref{smooth-energy} and with \eqref{enhanced-E}, and let the dissipation potentials 
$(\Psi_n)_n$  be  given either by  \eqref{Psi-n-bis}, 
or by \eqref{specific-1-limsup}. 
\par
Then, for every $u \in \BV ([0,T];\R^d)$ there exists a sequence 
$(u_n)_n \subset \AC ([0,T];\R^d)$, converging to $u$ in  the \emph{strict} topology of  $\BV ([0,T];\R^d)$, such that 
\begin{equation}
\label{limsup-ineq}
\limsup_{n\to \infty} \Jfu{\Psi_n, \calE}{u_n} \leq \Jfu{\psiri,\bptname,\calE}{u}.
\end{equation}
\end{theorem}
Clearly, 
Theorems  \ref{thm:liminf} and \ref{thm:limsupgenerale} (whose proof is also postponed to Section \ref{s:proofs}),   yield the \emph{Mosco-convergence} of the functionals $( \Jfuname{\Psi_n, \calE})_n$ to $ \Jfuname{\psiri,\bptname,\calE}$, with respect to the weak-strict topology of $\BV ([0,T];\R^d)$, in the vanishing-viscosity and stochastic  cases. 
\par
Another straightforward consequence of Theorem \ref{thm:limsupgenerale}, in the spirit of Corollary \ref{cor:liminf}, is the following \emph{reverse approximation} result. 
\begin{corollary}
Let $\calE$ comply with \eqref{smooth-energy} and with \eqref{enhanced-E}. 
Let  $\psiri(v) = A \| v \|_1 $ and  $  \bpt 0v{\xi} = \|v\|_1 (A \vee \|\xi\|_\infty)$.

Then, for every  Balanced Viscosity solution  $u \in \BV ([0,T];\R^d) $ to the rate-independent system
$\RIS$ there exists a sequence $(u_n)_n \subset \AC ([0,T];\R^d)$ of solutions to the gradient systems 
  $(\Psi_n,\calE)$,
  with the dissipation potentials $(\Psi_n)_n$ given 
   by 
 $\Psi_n(v)  = A \| v\|_1 + \frac{\eps_n}{2}
\| v\|_1^2$  for all $n\in\N$, with   $\eps_n \down 0$
as $n\to\infty$
  (\emph{$\psiri$-vanishing viscosity}),  
   such that $u_n \to u$ as $n\to\infty$ \emph{strictly} in $ \BV([0,T];\R^d)$. 
   \par
    A completely analogous statement holds with the dissipation potentials $(\Psi_n)_n$ from 
   \eqref{Psi-n-bis}. 
\end{corollary}

%

\section{Proofs  of Theorems \ref{thm:liminf} and \ref{thm:limsupgenerale}}
\label{s:proofs}
In what follows, we will denote by $C$ a generic positive constant independent of $n$, whose meaning may vary \RRR even within the same line.   
\par
 We will just outline the argument for the proof of \underline{\bf Proposition \ref{prop:compactness}}, referring to the argument for  \cite[Thm.\ 4.1]{MRS12} (see also \cite[Thm.\ 4.2]{BonaschiPeletier14}) for all details. 
Combining the information that $\Jfu{\Psi_n,\calE}{u_n} \leq C$ with
the power control condition from \eqref{smooth-energy}, we find that
\[
\int_0^T \left( \Psi_n (\dot u(s)) + \Psi_n^* (-\rmD \ene t{u(s)})\right) \dd s
+ \ene T{u_n(T)}  \leq  C+  \int_0^T C_1 | \ene s{u_n(s)} | \dd s,
\]
where we have also used that $\| u_n(0) \|\leq C$,
 and thus $\sup_n |\ene 0{u_n(0)}| \leq C$. 
  Taking into account that both $\Psi_n$ and $\Psi_n^*$
 are positive, via the Gronwall Lemma we deduce from the above inequality that
 $\sup_{t\in [0,T]}|\ene t{u_n(t)}|\leq C$, whence $\sup_{t\in [0,T]}|\partial_t\ene t{u_n(t)}|\leq C$.
 Hence
 \[
 \int_0^T \left( \Psi_n (\dot{u}_n(s)) + \Psi_n^* (-\rmD \ene t{u(s)})\right) \dd s \leq C,
 \]
 which implies thanks to \RRR the coercivity \EEE \eqref{h:2} that
 $\Var{}{u_n}0T\leq C$.  Then, the thesis readily follows from the Helly theorem.
\QED
\RRR Before developing the proof of Theorem \ref{thm:liminf},  we
preliminarily  give \EEE the following lower semicontinuity result, in the
spirit of \cite[Lemma 3.1]{MRS09}, cf.\ also
\cite[Lemma~4.3]{MRS12}).
\begin{lemma}
\label{lsci-lemma} Let  $m,\, d \geq 1$ and $\mathfrak{F}_n,\, \mathfrak{F}_\infty : \R^m
\times \R^d \to [0,+\infty)$ be normal integrands such that
\begin{enumerate}
\item for fixed $\xi \in \R^d$ the functionals
 $\mathfrak{F}_n(\cdot, \xi) $
are convex for every $n \in \N \cup \{ \infty\}$,
 \item there holds
\begin{equation}
\label{G-liminf} \Gamma\text{-}\liminf_{n \to \infty} \mathfrak{F}_n
\geq \mathfrak{F}_\infty \quad \text{in } \R^m \times \R^d.
\end{equation}
\end{enumerate} Let $I$ be a bounded interval in $\R$
and let $w_n, \, w  : I \to \R^m$ fulfill $w_n \weakto w $ in
$L^1(I;\R^m)$, and
 $\xi_n, \,  \xi: I \to \R^d$ fulfill
 $\xi_n(s) \to \xi(s)$ for almost all $s\in I$.
 Then
\begin{equation}
\label{Ioffe} \liminf_{n \to \infty} \int_I
\mathfrak{F}_n(w_n(s),\xi_n(s)) \dd s \geq \int_I
\mathfrak{F}_\infty(w(s),\xi(s)) \dd s.
\end{equation}
\end{lemma}
\begin{proof}
We introduce the functional
\[
\overline{\mathfrak{F}} : \N \cup \{ \infty\} \times \R^m \times
\R^d, \qquad \overline{\mathfrak{F}}(n,w,\xi): =
\begin{cases}
\mathfrak{F}_n (w,\xi) & \text{for } n \in \N,
\\
\mathfrak{F}_\infty (w,\xi) & \text{for } n =\infty.
\end{cases}
\]
It follows from \eqref{G-liminf} that $\overline{\mathfrak{F}}$ is
lower semicontinuous on  $ \N \cup \{ \infty\} \times \R^m \times
\R^d$, hence it is a positive normal integrand. Then, \eqref{Ioffe}
follows from the Ioffe Theorem, cf.\ \cite{Ioffe} and also, e.g.,
\cite[Thm.\ 21]{Valadier}.
\end{proof}

\paragraph{\underline{\bf Proof of Theorem \ref{thm:liminf}}.}
Let $ (u_n)_n  \subset \BV([0,T];\R^d)$ be a sequence weakly
converging to $u \in \BV([0,T],\R^d)$. We may suppose that
$\liminf_{n\to\infty} \Jfu{\Psi_n,\calE}{u_n}<+\infty$, as otherwise
there is nothing to prove. Therefore, up to a  subsequence we  have $
\Jfu{\Psi_n,\calE}{u_n} \leq C$, in particular yielding that $u_n
\in \AC ([0,T];\R^d)$ for every $n\in \N$. With the very same
arguments as in the proof of Prop.\ \ref{prop:compactness}, 
\RRR also based on the power control \eqref{smooth-energy}, \EEE
we see
that each contribution to $\Jfu{\Psi_n,\calE}{u_n} $ is itself
bounded. Convergences \eqref{energies+power}
 follow from the pointwise convergence of $(u_n)_n$, the fact that   $\calE\in \rmC^1 ([0,T]\times \R^d)$, and the  Lebesgue dominated convergence theorem, recalling that $(u_n)_n$ is bounded in $L^\infty (0,T;\R^d)$.
Moreover, we have  that $\rmD \ene t{u_n(t)}\to \rmD\ene t{u(t)}$  for every $t \in [0,T]$. Then, taking into account that  the functionals $(\Psi_n^*)_n$
$\Gamma$-converge to $\psiri^*$,
 we can apply  Lemma~\ref{lsci-lemma}  to the functionals $\mathfrak{F}_n(w,\xi) := \Psi_n^*(\xi)$ and
  $\mathfrak{F}(w,\xi):= \psiri^*(\bw)$ to obtain
\begin{equation}
\label{local-stability}
\begin{gathered} \liminf_{n \to \infty} \int_0^T
\Psi_n^*(-\rmD \ene t{u_n(t)})\dd t \geq \int_0^T \psiri^*(-\rmD\ene
t{u(t)}) \dd t, 
\\  \text{whence }  \quad -\rmD\ene t{u(t)} \in K^* \
\foraa\, t \in (0,T).
\end{gathered}
\end{equation}

Define the non-negative finite measures on $[0,T]$
\begin{equation*}
\nu_n :=  \Psi_n(\dot{u}_n(\cdot))\Leb 1 + \Psi_n^*(-\rmD \ene {\cdot}{u_n(\cdot)}) \Leb  1 \doteq \mu_n + \eta_n.
\end{equation*}
Up to extracting a subsequence, we can suppose that they weakly$^*$ converge to a positive measure
\[
\nu= \mu + \eta \qquad \text{with }  \eta \geq \psiri^*(-\rmD\ene {\cdot}{u(\cdot)}) \Leb 1.
\]
Let us now preliminarily show that \RRR
\begin{equation}
\label{to-show-mu}
\nu \geq  \psiri(\dot u) \Leb 1 + \mu_{\psiri,\Ca}. 
\end{equation}
For this, we shall in fact observe that $\mu \geq \psiri(\dot u) \Leb 1 + \mu_{\psiri,\Ca}$. 
This will follow upon proving that  
\begin{equation}
\label{Var-ineqs}
\mu([\alpha,\beta]) = \lim_{n\to\infty} \int_{\alpha}^\beta \Psi_n(\dot{u}_n(t)) \dd t \geq \Var{\psiri}{u}{\alpha}{\beta} \quad  \text{for every $[\alpha,\beta]\subset [0,T]$}. 
\end{equation}
Indeed, let us fix  a partition $t_0=\alpha < t_1 < \ldots < t_k = \beta$ of $[\alpha,\beta]$
and notice that 
\[
\begin{aligned}
\lim_{n\to\infty} \int_\alpha^\beta \Psi_n(\dot{u}_n(t)) \dd  t = \lim_{n\to\infty} 
\sum_{m=1}^k \int_{t_{m-1}}^{t_m} \Psi_n(\dot{u}_n(t)) \dd  t  
 & \stackrel{(1)}{\geq}  
\liminf_{n\to\infty} 
\sum_{m=1}^k (t_m - t_{m-1}) \Psi_n \left(  \frac{\int_{t_{m-1}}^{t_m} \dot{u}_n(t) \dd  t}{t_m - t_{m-1}}   \right)
\\
& = 
\liminf_{n\to\infty} 
\sum_{m=1}^k (t_m - t_{m-1}) \Psi_n \left( \frac{u_n(t_m) - u_n(t_{m-1})}{t_m - t_{m-1}}    \right)
\\
& 
\stackrel{(2)}{\geq} \sum_{m=1}^k (t_m - t_{m-1}) \psiri \left( \frac{u(t_m) - u(t_{m-1})}{t_m - t_{m-1}}    \right)
 \\
 & \stackrel{(3)}{=}  \sum_{m=1}^k \psiri \left( u(t_m) - u(t_{m-1})   \right),
\end{aligned}
\]
where (1) follows from the Jensen inequality, (2) from the fact that  the potentials $(\Psi_n)_n$
$\Gamma$-converge to $\psiri$ (cf.\ Lemma \ref{l:repr-psiri}), and (3) from the $1$-positive homogeneity of $\psiri$. Since the partition 
of $[\alpha,\beta]$ is arbitrary, we conclude \eqref{Var-ineqs}. 
\EEE
\par
However, we need to improve \eqref{to-show-mu}
by obtaining a finer characterization for the jump part of $\nu$. We will in fact prove that
\begin{equation}
\label{better-charact-jump}
\nu(\{ t\}) \geq   \Cost{\bptname}{\calE}{t}{u(t_-)}{u(\RRR t_+ \EEE)} 
 \qquad \text{for every } t \in \ju u
\end{equation}
by adapting the argument in the proof of \cite[Prop.\
7.3]{mielke-rossi-savare2013}. To this end, for fixed $t\in \ju u
$ let us pick two sequences $h_n^-\uparrow t$ and $h_n^+ \downarrow
t$ such that $u_n (h_n^-)\to u(t_-)$ and $u_n (h_n^+)\to u(t_+)$.
Define $\mathsf{s}_n : [ h_n^-, h_n^+] \to \R$ by
\begin{equation}
\label{def-sn}
\mathsf{s}_n(h) : = c_n \left(h -h_n^- + \int_{h_n^-}^{h} \left(  \Psi_n(\dot{u}_n(t)) + \Psi_n^*(-\rmD \ene {t}{u_n(t)})  \right) \dd t \right), \qquad 
h \in [h_n^-,h_n^+],
\end{equation}
where the  normalization constant  $c_n$ is chosen in  such a way  that $\mathsf{s}_n(h_n^+)=1$.
Therefore, $\mathsf{s}_n$ takes values in $[0,1]$. 
 Observe that for every $n$ the function $\mathsf{s}_n$ is strictly increasing and thus invertible, and let
\[
\mathsf{t}_n:= \mathsf{s}_n^{-1}: [0,1] \to  [h_n^- ,  h_n^+] \quad \text{and} \quad \teta_n: = u_n \circ \mathsf{t}_n.
\]
There holds
\begin{equation}
\label{normalization-estimate}
\dot{\mathsf{t}}_n(s) + \|\dot{\teta}_n\|_1(s) = \frac{1+ \|\dot{u}_n\|_1(\mathsf{t}_n(s))}{ c_n \left(1+ \Psi_n(\dot{u}_n(\mathsf{t}_n(s))) + \Psi_n^*(-\rmD \ene {\mathsf{t}_n(s)}{u_n(\mathsf{t}_n(s))})   \right)} \leq C \quad \foraa\, s \in (0,1).
\end{equation}
Now, by the upper semicontinuity property of the weak$^*$-convergence of measures on closed sets we have
\begin{equation}
\label{inverses}
\begin{aligned}
\nu(\{ t\}) \geq \limsup_{n\to\infty} \nu_n([h_n^-,h_n^+])  & \geq \liminf_{n\to\infty} \int_{h_n^-}^{h_n^+}
\left(  \Psi_n(\dot{u}_n(t)) + \Psi_n^*(-\rmD \ene {t}{u_n(t)})  \right) \dd t
\\
&
\stackrel{(1)}{=} \liminf_{n\to\infty}\int_0^1
\left(  \Psi_n(\dot{u}_n(\mathsf{t}_n(s))) + \Psi_n^*(-\rmD \ene {\mathsf{t}_n(s)}{u_n(\mathsf{t}_n(s))})  \right) \dot{\mathsf{t}}_n(s)\dd s
\\
&
\stackrel{(2)}{=}\liminf_{n\to\infty}\int_0^1
\bip{\Psi_n}{\dot{\mathsf{t}}_n(s)}{\dot \teta_n(s)}{-\rmD \ene {\mathsf{t}_n(s)}{\teta_n(s)}} \dd s
\end{aligned}
\end{equation}
where (1) follows  from a change of variables, and (2) from the very  definition \eqref{def-bptreg}  of $\bipname{\Psi_n}$. Now, it follows from
\eqref{normalization-estimate} and from the fact that the range of $\mathsf{t}_n$ is $ [h_n^- ,  h_n^+] $  that there exists  $(\mathsf{t},\teta) \in \mathrm{C}_{\mathrm{lip}}^0 ([0,1]; [0,T]\times \R^d)$  such that,  up to a not relabeled subsequence,
\begin{equation}
\label{good-convergences}
\begin{gathered}
\mathsf{t}_n(s)\to \mathsf{t}(s) \equiv t, \quad \teta_n(s) \to \teta(s)  \text{ for all } s\in [0,1], \qquad \dot{\mathsf{t}}_n \weaksto 0  \text{ in } L^\infty (0,1), \quad
 \dot{\teta}_n \weaksto \dot{\teta}  \text{ in } L^\infty (0,1;\R^d), 
\\
\text{so that } \teta(0) = \lim_{n\to\infty} u_n(h_n^-)  = u(t_-)
\text{ and } \teta(1) = \lim_{n\to\infty} u_n(h_n^+)  = u(t_+)\,.
\end{gathered}
 \end{equation}
 Therefore, applying Lemma \ref{lsci-lemma}  above with the choices $m=d+1$ and, for $w=(\tau,v) \in \R \times
 \R^d$, with
 $\mathfrak{F}_n(w, \xi) = \mathfrak{F}_n(\tau,v, \xi) := \bip{\Psi_n}{\tau}{v}{\xi}$ and $\mathfrak{F}_\infty (w,\xi): = \bpt{\tau}{v}{\xi}$
 (where we still denote by $\bipname{\Psi_n}$ and by $\bptname$ their  t extensions  \RRR  to $\R \times \R^d \times \R^d$ by infinity), \EEE and taking into account
 \eqref{def-p-gliminf} from Hyp.\ \ref{hyp:p-lim}, which ensures the validity of \RRR condition 
 \eqref{G-liminf} in Lemma \ref{lsci-lemma}, \EEE we conclude
\[
\liminf_{n\to\infty}\int_0^1
\bip{\Psi_n}{\dot{\mathsf{t}}_n(s)}{\dot \teta_n(s)}{-\rmD \ene {\mathsf{t}_n(s)}{\teta_n(s)}} \dd s
 \geq \int_0^1
\bpt{0}{\dot \teta(s)}{-\rmD \ene {t}{\teta(s)}} \dd s  \geq  \Cost{\bptname}{\calE}{t}{u(t_-)}{\RRR u(t_+) \EEE}\,.
\]
Similarly, we prove
that
\[
\limsup_{n\to\infty} \nu_n ([h_n^-,t]) \geq \Cost{\bptname}{\calE}{t}{u(t_-)}{u(t)}, \qquad \limsup_{n\to\infty} \nu_n ([t,h_n^+]) \geq \Cost{\bptname}{\calE}{t}{u(t)}{u(t_+)}.
\]

Repeating the very same arguments as in the proof of \cite[Prop.\ 7.3]{mielke-rossi-savare2013}, we ultimately find that 
\[
\liminf_{n\to\infty}\int_s^t \left( \Psi_n(\dot{u}_n(r)) + \Psi_n^* (-\rmD \ene r{u_n(r)})\right) \dd r \geq \Var{\psiri,\bptname,\calE}{u}{s}t \quad \text{for every $0 \leq s \leq t$},
\]
whence \eqref{liminf_better} also in view of \eqref{local-stability}. \RRR This concludes the proof. 
 \QED 
\paragraph{\bf Proof of Corollary \ref{cor:liminf}.}
Let $u\in \BV ([0,T];\R^d)$ be a limit point of the sequence $(u_n)_n \subset \AC ([0,T];\R^d)$. It follows from \eqref{liminf-ineq} that 
$ \Jfu{\psiri,\bptname,\calE}{u}=0$, hence by Prop.\ \ref{prop:null-minim} $u$ is  a Balanced Viscosity solution to $\RIS$. Moreover,
for every $0\leq s \leq t \leq T$ we have that
\[
\begin{aligned}
&
\limsup_{n\to\infty} \int_s^t \left( \Psi_n(\dot{u}_n(r)) + \Psi_n^* (-\rmD \ene r{u_n(r)})\right) \dd r   
\\
 & \stackrel{(1)}{\leq} \limsup_{n\to\infty}  
\left( \ene{s}{u_n(s)} - \ene{t}{u_n(t)} + \int_s^t \partial_t \ene {r}{u_n(r)}  \dd r + \eps_n \right)
  \\
 & \stackrel{(2)}{=} 
 \ene{s}{u(s)} - \ene{t}{u(t)} + \int_s^t \partial_t \ene {r}{u(r)}  \dd r 
 \\
 & \stackrel{(3)}{=} 
   \Var{\psiri,\bptname,\calE}{u}{s}t +  \int_s^t  \psiri^* (-\rmD \ene t{u(r)}) \dd r
\end{aligned}
\]
where (1) follows from $\Jfu{\Psi_n}{u_n} \leq \eps_n $, (2) from convergences \eqref{energies+power}, and (3) from the fact that $ \Jfu{\psiri,\bptname,\calE}{u}=0$. 
Combining this with \eqref{liminf_better},  we conclude
the enhanced convergences 
 \eqref{lim-better}. 
\QED
\EEE
\paragraph{\underline{\bf Proof of Theorem \ref{thm:limsupgenerale}}.}
Given $u \in \BV([0,T],\R^d)$, we will construct a sequence  \RRR $(u_n)_n\subset \AC([0,T];\R^d)$ \EEE such that 
 $  u_n\to u $ strictly in $\BV([0,T];\R^d)$ and \EEE
\begin{equation}
\label{to-prove-indeed}
\limsup_{n\to\infty} \J_{\Psi_n,\E} (u_n) \leq \Jfu{\Psi_0,\bptname,\E}{u}.
\end{equation}
We split the proof of in several steps;  for Steps 1--4, we shall adapt the arguments from the 
 proof of 
\cite[Thm.\ 4.2]{BonaschiPeletier14}. 

\medskip

\noindent
\textit{Step $1$: reparametrisation.}
First we reparametrise the curve $ u$, in terms of a new time-like parameter $s$ on a domain $[0,S]$. The aim is to expand the jumps in $u$ into smooth connections. \RRR Following  \cite[Prop.~6.9]{MRS10}, \EEE we define
\[
\ss(t):= t + \Var{\psiri,\bptname,\calE}{u}{0}t.
\] 
Then there exists a Lipschitz parametrisation $(\st,\su):[0,S]\to[0,T]\times \R \;$ such that $\st$ is non-decreasing,
\begin{equation}
\label{eq:inverserelation}
\st(\ss(t))= t \qquad \text{and} \qquad 
\su(\ss(t))=u(t) \text{ for every } t \in [0,T],
\end{equation}
and such that
\begin{equation}
\label{eq:equivalenceparametrization}
\int_0^S \RRR \bptname (\dot \st(s),\dot \su(s),-\rmD \E ( \st(s),\su(s)) \dd s \EEE =  \Var{\psiri,\bptname,\calE}{u}{0}T +  \int_0^T  \psiri^* (-\rmD \ene t{u(t)}) \dd t\,.
\end{equation} 
Moreover, it also holds that
\begin{equation}
\label{eq:equivalenceVar}
    \Var{\psiri}\su 0S =    \Var{\psiri}u0T .
\end{equation}
\medskip

\noindent
\textit{Step $2$: preliminary remarks.}  \RRR Since we will construct a sequence $(u_n)_n$ strictly (and in particular pointwise)  converging to $u$ in $\BV([0,T];\R^d)$, 
thanks to the smoothness  of $\calE$ (cf.\ \eqref{smooth-energy}), we will have  
for the first three contributions to  $\J_{\Psi_n,\E}(u_n)$ 
 \EEE
\[
\E(T,  u_n (T)) - \E(0,u_n(0)) - \int_0^T \partial_t \E( t, u_n(t))\dd t \to  \E(T,  u (T)) - \E(0,u(0)) - \int_0^T \partial_t \E( t, u(t))\dd t
\]
as $n\to \infty$.  \RRR Therefore, in order to prove  \eqref{to-prove-indeed} it will be sufficient to \EEE
focus on the other terms in $\J_{\Psi_n,\E}$ and $\J_{\Psi_0,\bptname,\E}$. In view of~\eqref{eq:equivalenceparametrization},  it will be  sufficient to prove that
\begin{equation}
\label{limsup-thm:ineq-to-prove}
\limsup_{n\to\infty} \int_0^T \left[ \Psi_n \left( \dot u_n (t) \right) + \Psi_n^*\left( \rmD \E(t, u_n(t)) \right) \right] \dd t 
\leq \int_0^S \bptname (\dot \st(s),\dot \su(s),\rmD \E ( \st(s), \su(s))) \dd s.
\end{equation}
\medskip

\noindent 
\textit{Step $3$: definition of the new time $\st_n$ and  of the recovery sequence $u_n$.}
For the  sake of simplicity, in  what follows we construct a recovery sequence  for a curve $u$ \RRR \emph{with jumps only at $0$ and $T$}, postponing to the 
end of the proof  (cf.\ \textit{Step $7$}), the discussion of the case of a curve with countably many jumps. 
We define $u_n$ \EEE by first perturbing the time variable $\st$: we  fix $\delta > 0$ \RRR  and consider a selection
\begin{equation}
\label{selection-tau-n-d}
\RRR \tau_n^\delta (s)   \in  \argmin_{\tau>0} \bipname{\Psi_n}^\delta(\tau,\dot \su(s) ,- \rmD \E ( \st(s), \su(s))). \EEE
\end{equation}
We define $\st_n:[0,S]\to[0,T_{n}]$ as the solution of the differential equation
\begin{equation}
\label{def-stn}
\st_n(0) = 0, \qquad \dot{\st}_{n}(s)=  \dot{\st}(s) \vee \tau_{n}^\delta (s)\,.
\end{equation}
\RRR Observe that $\dot{\st}(s) =0$ in $  [0,\ss(0^+)] \cup \in [\ss(T^-),S]$, but we can assume that 
$|\dot \su (s)| > 0$  on $  [0,\ss(0^+)] \cup \in [\ss(T^-),S]$. This will be sufficient to guarantee that $ \argmin_{\tau>0} \bipname{\Psi_n}^\delta(\tau,\dot \su , \rmD \E ( \st, \su)) \neq \emptyset$ on the latter set, and thus that $\tau_n^\delta$ is well defined.  \EEE
On the other hand, 
 for $s \in [\ss(0^+),\ss(T^-)]$, we have  
$
\dot \st (s) = \left. \frac{1}{\dot \ss(t)} \right|_{t=\st(s)} > 0.
$
All in all,  $\dot\st_n(s)>0$ for all $s \in [0,S]$.  The range of $\st_n $ is $[0,T_n]$, with $T_n\geq T$; since the recovery sequence $u_n$ has to be defined on the interval $[0,T]$, we rescale $\st_n$ by
\begin{equation}
\label{lambda-n}
\lambda_n := \frac{T_n}T \geq 1,
\end{equation}
and define our recovery sequence as follows: 
\begin{equation}
\label{def-un}
u_{n}(t):=\su \left( \st_{n}^{-1} \left( t\lambda_n \right) \right), \qquad \text{ so that }\qquad 
\dot u_{n}(t)=\frac{\dot{\su}}{\dot{\st}_{n}}\left(  \st_{n}^{-1} \left( t \lambda_n  \right) \right)\lambda_n .
\end{equation}
Now we substitute the explicit formula for $u_n$, we perform a change of variable and obtain
\[
\begin{aligned}
 & \int_0^T \Bigl[ \Psi_{n}\left( \dot u_{n}(t) \right) + \Psi_{n}^*\left(t,- \rmD \E(t, u_{n}(t)) \right) \Bigr]\dd t \\
 & = \int_0^T \left[ \Psi_{n}\left( \frac{\dot{\su}}{\dot{\st}_{n}} \left( \st_{n}^{-1} \left( t\lambda_{n} \right)\right) \lambda_{n}  \right) + \Psi_{n}^*\left( -\rmD \E\bigl(t,\su ( \st_{n}^{-1} ( t\lambda_{n} ))\bigr) \right) \right] \dd t \\
 & =\int_0^S\left[ \Psi_{n} \left( \frac{\dot{\su}}{\dot{\st}_{n}}(s)\lambda_{n} \right) + \Psi_{n}^* \Bigl(- \rmD \E(\st_{n}(s)\lambda_{n}^{-1}, \su(s)) \Bigr) \right]\frac{\dot{\st}_{n}(s)}{\lambda_{n}}\dd s ,
\end{aligned}
\]
so that
\[
\int_0^T \Bigl[ \Psi_{n}\left( \dot u_{n}(t) \right) + \Psi_{n}^*\left( - \rmD \E(t, u_{n}(t)) \right) \Bigr] \dd t = \int_0^S \bip{\Psi_n}{\lambda_n^{-1} \dot{\mathsf{t}}_n(s)}{\dot{\su}(s)}{-\rmD \ene {\mathsf{t}_n(s)\lambda_n^{-1}} {\su(s)}}\dd s.
\]
\medskip

\noindent
\textit{Step $4$: Strict convergence of $(u_n)_n$.}
\RRR Recall that we need to  prove the pointwise convergence $u_n(t) \to u(t)$ for all $t \in [0,T]$ and the convergence of the variations.
For this, it will be crucial to have the following property, that shall be verified  (even uniformly w.r.t.\ $s\in [0,S]$) both in the
stochastic (cf.\ \eqref{conse-est-tau-delta}), and in the 
 vanishing-viscosity cases 
(cf.\ \eqref{eq:explicit_tau_vv}):
\begin{equation}
\label{ripristinata}
 \tau_n^\delta(s)  \to 0 \quad \text{as } n \to \infty \quad \foraa\, s \in (0,S).
\end{equation}
 This implies that $\dot{\st}_{n}(s) \to \dot{\st}(s)$ for almost all $s\in (0,S)$, \EEE and then it   will  hold 
\[
\st_{n}(s) \to \st(s) \quad  \text{for every }  s \in [0,S]  \qquad \implies \qquad \lambda_n \to 1,  \quad 
\st_{n}^{-1}(t \lambda_{n}) \to \ss(t) \quad \text{for every } t \in [0,T].
\]
Moreover, $\dot \st_n(s) > 0$ implies that $\st_n^{-1} (0) = 0$ and $\st_n^{-1} (T_n) = S$, and so we  will  have   the desired pointwise convergence 
\[
u_{n}(t)= \su \left( \st_{n}^{-1} \left( t \lambda_{n} \right)  \right) \to \su(\ss(t))\stackrel{\eqref{eq:inverserelation}}= u(t) \qquad \text{for every } t \in [0,T].
\]
The convergence of the variations  will be  automatic, since by definition of \RRR  $u_{n}$ \EEE we will have 
 \begin{equation*}
\int_0^T A \|\dot u_{n}(t)\|_1 \dd t = \int_0^S A \|\dot{\su}(s)\|_1 \dd s  = \Var {\Psi_0} \su 0 S \stackrel{\eqref{eq:equivalenceVar}}{=} \Var {\Psi_0} u 0 T .
\end{equation*} 
 Therefore, from now on we will concentrate on the proof of the $\limsup$ estimate \eqref{limsup-thm:ineq-to-prove}. 
 \medskip

\noindent
\textit{Step $5$: strategy for  \eqref{limsup-thm:ineq-to-prove}.}   First of all, we will  show  the following \emph{pointwise} $\limsup$-inequality \EEE
\begin{equation}
\label{lemma:thm-limsup-estimates}
\limsup_{n \to \infty} \bip{\Psi_n}{\lambda_n^{-1} \dot{\mathsf{t}}_n(s)}{\dot \su(s)}{-\rmD \ene {\mathsf{t}_n(s)\lambda_n^{-1}}{\su(s)}} \leq
\RRR  \bptname (\dot \st(s),\dot \su(s), - \rmD \E ( \su(s),\st(s)) \EEE \  \ \foraa\, s \in (0,S)\,.
\end{equation}
 \RRR Secondly, we will  apply the following 
version of the Fatou Lemma
\begin{equation}
\label{Fatou}
\left.
\begin{array}{rrr}
 & \limsup_{n\to \infty} f_n(s) \leq f(s)  &  \foraa\, s \in (0,S),
\\
 & f_n(s) \leq g_n(s)  &  \foraa\, s \in (0,S),
\\
& 
g_n \to g & \text{in } L^1(0,S),
\end{array}
\right\} 
\ \Longrightarrow \ \limsup_{n\to\infty} \int_0^S f_n(s) \dd s \leq   \int_0^S f(s) \dd s,
\end{equation}
for measurable functions $(f_n)_n$ and $f$, in order to conclude 
\begin{equation}
\label{fatou-applied}
\limsup_{n\to\infty} \int_0^S \bip{\Psi_n}{\lambda_n^{-1} \dot{\mathsf{t}}_n(s)}{\dot{\su}(s)}{-\rmD \ene {\mathsf{t}_n(s)\lambda_n^{-1}} {\su(s)}}\dd s 
\leq \int_0^S \bptname (\dot \st(s),\dot \su(s), - \rmD \E ( \su(s),\st(s))) \dd s,
\end{equation}
whence \eqref{limsup-thm:ineq-to-prove} and ultimately \eqref{to-prove-indeed}. 
For the proof of \eqref{lemma:thm-limsup-estimates} and \eqref{fatou-applied}, we will distinguish the \emph{stochastic}  and the \emph{vanishing-viscosity} cases. 
 \EEE
\medskip

\noindent
\RRR \textit{Step $6a$: proof of 
 \eqref{lemma:thm-limsup-estimates} and \eqref{fatou-applied} for  $\Psi_n$ given by \eqref{Psi-n-bis} (stochastic approximation).} \EEE
 Preliminarily, we observe that,
 with the very same calculations as for 
\eqref{estimate_tau_n_delta} (cf.\ Claim 3.4 in the proof of 
Proposition \ref{l:large-deviations}),  one has 
\begin{equation}
\label{conse-est-tau-delta}
\begin{aligned}
   &  \tau_n^\delta(s)  
  \leq 
   \sqrt{ \frac{d^2\|\dot{u}(s)  \|\f^2}{n^2 \delta^2 + n^2 \left(\Psi_n^*(-\rmD \ene{\st(s)}{\su(s)})\right)^2} }  \to 0 
   \quad \text{for almost all } s \in (0,S),
   \quad \text{and thus}
   \\
   &
\sup_{s\in [0,S]}\tau_n^\delta(s)  \leq  \frac{C}{\delta n} \to 0    \quad \text{as } n\to\infty,
\end{aligned}
\end{equation}
(with a slight abuse of notation, we use the symbol $\sup$ also for an essential supremum)
where we  have exploited the Lipschitz continuity of $u$. 
 In order to prove the pointwise inequality \eqref{lemma:thm-limsup-estimates}, 
we start with the following algebraic manipulation
\begin{equation}
\label{manipulation}
\begin{split}
\bip{\Psi_n}{\lambda_n^{-1} \dot{\mathsf{t}}_n(s)}{\dot \su (s)}{-\rmD \ene {\mathsf{t}_n(s)\lambda_n^{-1}}{\su(s)}} 
=  &\; \bipname{\Psi_n}^\delta(\tau_n^\delta(s) ,\dot \su (s), - \rmD \E ( \st(s), \st(s))) - \tau_n^\delta(s) \delta\\ 
 & +  \dot \st_n (s)\Psi_n^*( -\rmD \E (\st_n(s)\lambda_n^{-1},  \su(s))) - \tau_n^\delta(s) \Psi_n^*(- \rmD \E ( \st(s), \su(s))) \\
 & +  \frac{\dot \st_n(s)}{\lambda_n} \Psi_n \left( \frac{\dot \su(s)}{\dot \st_n(s)}\lambda_n \right) - \tau_n^\delta(s) \Psi_n \left( \frac{\dot \su(s)}{\tau_n^\delta(s)} \right)
\end{split}
\end{equation}
 and prove the following three claims for   the terms on the right-hand side.  
\\
 \textit{Claim $6.a.1$: there holds}
\begin{equation}
\label{1st-ingredient-manip}
\limsup_{n \to \infty}\bipname{\Psi_n}^\delta(\tau_n^\delta(s),\dot \su(s) ,  - \rmD \E ( \st(s), \su(s)))   - \tau_n^\delta(s) \delta \leq \bptname(\dot \st(s),\dot \su (s),- \rmD \E ( \st(s), \su(s)) ) \qquad \foraa\, s \in (0,S),
\end{equation}
\textit{with $\bptname$ given by \eqref{p-lim-van-visc-bis}.} Indeed, the representation formula \eqref{supp-re} for   $\min_{\tau>0}\bipnamed{\Psi_n}\delta$  
and estimate \eqref{eq:bound_xin} (cf.\ Claim 3.3 in the proof of Prop.\ \ref{l:large-deviations})  yield
\begin{equation}
\label{to-cite-later}
\begin{split}
\bipname{\Psi_n}^\delta(\tau_n^\delta(s),\dot \su(s) , -  \rmD \E ( \st(s), \su(s))) & = \sup \left\{ \langle \xi, \dot \su(s)  \rangle  \; | \; \xi \in K_{n,\delta}^*(-\rmD \E(\st(s),\su(s))) \right\} \\ & \leq \sup \left\{ \| \dot \su(s)  \|_1 \| \xi \|\f \; | \; \xi \in K_{n,\delta}^*(-\rmD \E(\st(s),\su(s))) \right\} \\
& \leq \| \dot \su(s) \|_1 \left( A \vee \|  \rmD \E(\st(s),\su(s))  \|\f \right) + \frac1n \| \dot \su(s) \|_1 \log(2en\delta),
\end{split}
\end{equation}
and we conclude sending $n\to\infty$. 
Furthermore, we 
observe that  $\tau_n^\delta(s) \delta \to 0$ as  $n \to \infty$ thanks to the previously proved  \eqref{conse-est-tau-delta}. \EEE 
\\
 \textit{Claim $6.a.2$: there holds}
\begin{align}
\label{est_reg_case_1} & \limsup_{n \to \infty} \dot \st_n(s) \Psi_n^*( - \rmD \E (\st_n(s)\lambda_n^{-1}, \su(s))) - \tau_n^\delta(s) \Psi_n^*(- \rmD \E ( \st(s), \su(s))) \leq 0 && \foraa\, s \in (0,S).
\end{align}
\RRR Indeed,   from the uniform Lipschitz continuity of $\rmD \calE(\cdot, u)$ (cf.\ \eqref{enhanced-E}), we gather that 
\begin{equation}
\label{2n-ingredient}
\begin{aligned}
|\rmD_i \E ( \st_n(s) \lambda_n^{-1}, \su(s)) |   -  |\rmD_i \E  (\st(s),\su(s)) |
 & \leq 
\left| \rmD_i \E ( \st_n(s) \lambda_n^{-1}, \su(s))  - \rmD_i \E  (\st(s),\su(s)) \right| 
\\
 & \leq  L_{\mathsf{E}} | \st_n(s) \lambda_n^{-1} - \st(s)|
\end{aligned}
\end{equation}
for all $ i =1, \ldots, d $.
We now observe that 
\begin{equation}
\label{evvaiGio}
\begin{aligned}
| \st_n(s) \lambda_n^{-1} - \st(s)| & \leq | \st_n(s) \lambda_n^{-1} - \st(s) \lambda_n^{-1} + \st(s) \lambda_n^{-1} - \st(s)|
\\
& \leq \st(s) (1 - \lambda_n^{-1}) + \lambda_n^{-1} | \st_n(s) - \st(s) |
\\
& = T \left( 1 - \frac{1}{\lambda_n} \right) + \lambda_n^{-1}  \int_0^{s} \left(  \dot{\st}(r) \vee \tau_{n}^\delta (r) - \dot{\st}(r) \right) \dd r 
\\
 &   \leq    T \left( 1 - \frac{1}{\lambda_n} \right) +  \int_0^{s} \tau_n^\delta(r) \dd r 
\end{aligned}
\end{equation}
where we have used the fact that $\lambda_n \geq 1$, the definition of $\st_n$ \eqref{def-stn}, 
and the bound on $\sup_{s\in[0,S]} \dot{\st}(s)$, since $\st$ is Lipschitz continuous.
We also  have   
 \begin{equation}
 \label{ln-returned}
 \begin{aligned}
 T (1-\lambda_n^{-1})   =  \frac{T}{T_n} (T_n-T)  &  \leq  \left( \int_0^{\ss(0^+)} \tau_n^\delta(r) \dd r  + \int_{\ss(0^+)}^{\ss(T^-)} \left( \dot \st(r) \vee \tau_n^\delta (r) - \dot \st(r) \right) dr+  \int_{\ss(T^-)}^S  \tau_n^\delta(r) \dd r \right)
 \\
 & \leq  \left( \int_0^S  \tau_n^\delta(r) \dd r \right) \leq  \sup_{s\in [0,S]}  \tau_n^\delta(s)
 \end{aligned}
 \end{equation}
 again using the definition \eqref{def-stn} of $\mathsf{t}_n$.
 Hence, combining estimate \eqref{2n-ingredient}
 with \eqref{evvaiGio} and 
\eqref{ln-returned}, we gather
that
\begin{equation}
\label{used-later-yes}
\begin{aligned}
&
|\rmD_i \E ( \st_n(s) \lambda_n^{-1}, \su(s)) |   -  |\rmD_i \E  (\st(s),\su(s)) | \leq C   \sup_{s\in [0,S]}  \tau_n^\delta(s) 
\doteq  \bar{C}(n) \quad \text{for all } s \in [0,S], \quad \text{with }
\\
&
\sup_{n\in \N}  n  \bar{C}(n) \doteq \overline{C}<\infty,
\end{aligned}
\end{equation}
 the latter estimate due to  \eqref{conse-est-tau-delta}. 
  Therefore, 
using now the explicit formula \eqref{def:cosh_psi*} for $\Psi_n^*$ we get   for almost all $s\in (0,S)$ that 
\begin{equation}
\label{spli}
\begin{split}
 & \dot \st_n(s) \Psi_n^*(- \rmD \E ( \st_n (s)\lambda_n^{-1}, \su(s))) - \tau_n^\delta(s) \Psi_n^*(- \rmD \E ( \st(s), \su(s)))\\  &
\stackrel{(1)}{ \leq} \frac{\dot \st_n(s)}{n} e^{-nA}\sum_{i=1}^d
\cosh ( n |\rmD_i \E ( \st_n(s) \lambda_n^{-1}, \su(s)) |  )
\\
&
\stackrel{(2)}{ \leq} \frac{\dot \st_n(s)}{n} e^{-nA}\sum_{i=1}^d
 \cosh ( n |\rmD_i \E (\st(s),\su(s))| + n \overline{C}(n)) 
 \\
&  \stackrel{(3)}{\leq} \frac{d \dot \st_n(s) e^{-nA}}{2n} + \frac{\dot \st_n(s)}{2n} e^{\overline{C}} e^{-nA} \sum_{i=1}^d e^{n |\rmD_i \E (\st(s),\su(s))|} 
\\
&
\stackrel{(4)}{\leq}
\begin{cases}
\displaystyle \frac dn \dot \st_n(s)  \left(1 + e^{\overline{C}} \right)  \doteq \frac Cn   & \text{ for } s \in [\ss(0^+),\ss(T^-)], \\
\displaystyle   C \left(\frac1n + \sup_{s \in [0,S] }\tau_n^\delta(s) \Psi_n^*(-\rmD\E (\st(s),\su(s))) \right) & \text{ for } s \in [0,\ss(0^+)) \cup (\ss(T^-),S], \\
\end{cases}
\end{split}
\end{equation}
where (1) follows from the positivity of $\Psi_n^*$ and from the trivial inequality $\cosh(nx) - 1 \leq \cosh(n|x|)$, 
(2) from \eqref{2n-ingredient}, (3) from \eqref{used-later-yes},
 using that $\cosh(x) \leq \frac{e^x+1}{2}$ for all $x\geq 0$, 
 and (4) is due to the fact that 
 $\| \rmD \E (\st(s), \su(s))\|\f \leq A $ for $s\in  [\ss(0^+),\ss(T^-)] $, and to an elementary inequality on $[0,\ss(0^+)) \cup (\ss(T^-),S]$.
 Clearly, $\frac C n  \to 0$; on the other hand,  it follows again from \eqref{conse-est-tau-delta}  and  the Lipschitz continuity of $u$    that 
 \begin{equation}
 \label{again-conse}
  \sup_{s \in [0,S] }\tau_n^\delta(s) \Psi_n^*(-\rmD\E (\st(s),\su(s))) 
   \leq  \sup_{s \in [0,S] } \frac{d\|\dot{u}(s)  \|\f   \Psi_n^*(-\rmD\E (\st(s),\su(s)))  }{ n \left(\Psi_n^*(-\rmD \ene{\st(s)}{\su(s)})\right)}  \leq \frac Cn 
  \to 0 \quad \text{as } n \to \infty.
 \end{equation}
 Therefore, 
 \eqref{est_reg_case_1}  ensues.
\\
 \textit{Claim $6.a.3$: there holds}
 \begin{align}
\label{est_reg_case_2} & \limsup_{n\to \infty}\frac{\dot \st_n(s)}{\lambda_n} \Psi_n \left( \frac{\dot \su(s)}{\dot \st_n(s)}\lambda_n \right) - \tau_n^\delta(s) \Psi_n \left( \frac{\dot \su(s)}{\tau_n^\delta(s)} \right) \leq 0 && \foraa\, s \in (0,S).
\end{align}
We use the explicit formula \eqref{Psi-n-bis} for $\Psi_n$, obtaining
\[
\begin{split}
 & \frac{\dot \st_n(s)}{\lambda_n}\Psi_n\left( \frac{\dot \su(s)}{\dot \st_n(s)}\lambda_n \right)  
 \\
& \stackrel{\dot \st_n \geq \tau_n^\delta}{\leq} \frac{\tau_n^\delta(s)}{\lambda_n}\Psi_n\left( \frac{\dot \su(s)}{\tau_n^\delta(s)}\lambda_n \right)  \\
&   \leq   \frac{d \tau_n^\delta(s) \, e^{-n A}}{n } +  \sum_{i=1}^d
\left[
 \frac{\dot \su_i(s)}{n}\log \left( \lambda_n \frac{ \frac{\dot \su_i(s)}{ \tau_n^\delta(s)} + \sqrt{ \left(\frac{\dot \su_i(s)}{\tau_n^\delta(s)}\right)^2 +  \frac{e^{-2n A}}{\lambda_n^2} }}{e^{-n A}} \right) - \frac{1}{n}\sqrt{\dot \su_i(s)^2 + \left( \frac{ \tau_n^\delta(s) \, e^{-n A}}{\lambda_n}\right)^2 } \right]  \\
&    \stackrel{\lambda_n \geq 1}{\leq}   \tau_n^\delta(s) \Psi_n\left( \frac{\dot \su(s)}{\tau_n^\delta(s)} \right) + \sum_{i=1}^d\left[ \frac{\dot \su_i(s)}{n} \log(\lambda_n)- \frac{1}{n}\sqrt{\dot \su_i(s)^2 + \left( \frac{ \tau_n^\delta(s) \, e^{-n A}}{\lambda_n}\right)^2 } + \frac{1}{n}\sqrt{ \dot \su_i(s)^2 + \left( \tau_n^\delta(s) \, e^{-n A} \right)^2} \right] \; ,
\end{split}
\]
 for almost all $s\in (0,S)$,  whence
\begin{equation}
\label{cited-later-3}
\begin{aligned}
& 
 \dot \st_n(s) \frac{1}{\lambda_n}\Psi_n \left( \frac{\dot \su(s)}{\dot \st_n(s)}\lambda_n \right) - \tau_n^\delta(s) \Psi_n \left( \frac{\dot \su(s)}{\tau_n^\delta(s)} \right) 
 \\
 & 
  \leq \sum_{i=1}^d \frac{\dot \su_i(s)}{n} \log(\lambda_n)- \frac{1}{n}\sqrt{(\dot \su_i(s))^2 + \left( \frac{ \tau_n^\delta(s) \, e^{-n A}}{\lambda_n}\right)^2 } + \frac{1}{n}\sqrt{ (\dot \su_i(s))^2 + \left( \tau_n^\delta(s) \, e^{-n A} \right)^2}\,.
  \end{aligned}
\end{equation}
\RRR Observe that the right-hand side of 
\eqref{cited-later-3} tends to zero as $n\to\infty$ taking into account that 
 $
 \sup_{s\in [0,S]}\| \dot \su(s) \|\f \leq C$, that  $\lambda_n \to 1$, and  that $\sup_{s\in[0,S]}\tau_n^\delta(s) \to 0$  by \eqref{conse-est-tau-delta}.  This yields 
\eqref{est_reg_case_2} and, ultimately,  \eqref{lemma:thm-limsup-estimates}.\EEE
\par
\RRR Finally, we conclude the integrated $\limsup$-estimate \eqref{fatou-applied} by observing that the Fatou Lemma (cf.\ \eqref{Fatou}) applies:
this can be checked combining \eqref{conse-est-tau-delta}, 
\eqref{manipulation},  \eqref{to-cite-later} (taking into account that $\sup_{s\in[0,S]} \| \dot \su(s) \|_1 \leq C$),   \eqref{spli}, \eqref{again-conse}, and \eqref{cited-later-3}. \EEE
\medskip

\noindent 
\RRR \textit{Step $6b$:  proof of 
 \eqref{lemma:thm-limsup-estimates} and \eqref{fatou-applied} for  $\Psi_n$ given by \eqref{specific-1-limsup} (vanishing-viscosity approximation).} \EEE
\RRR
To simplify the notation, in what follows we shall focus on the particular case
\[
\eps_n = \frac1n\,.
\]
 Preliminarily,
we recall that, in the case \eqref{specific-1-limsup}, 
\begin{equation}
\label{explicit-psi_n}
\Psi_n^*(\xi) = \frac{1}{2\eps_n} ( \min_{\zeta \in K^*} \| \xi - \zeta \|_* )^2 =\frac{n}{2} d_*(\xi,K^*)^2,
\end{equation}
where $\| \cdot \|_*$ is the dual norm to $\| \cdot \|$, and $d_*(\cdot, K^*)$ denotes the induced distance from the set $K^*$. 
Taking into account \eqref{explicit-psi_n}, 
 we provide  a bound for $\tau_n^\delta$ from \eqref{selection-tau-n-d} again resorting to the Euler-Lagrange equation 
 \eqref{Euler-Lagrange-tau}. In the present case, it gives
%
\begin{equation}
\label{eq:explicit_tau_vv}
( \tau_n^\delta(s))^2 = \frac{  \| \dot \su(s) \|^2}{2n \delta + n^2 d_*(-\rmD \ene{\st(s)}{\su(s)},K^*)^2}
\qquad \foraa\, s \in (0,S)\,.
\end{equation}
  In what follows we will take 
\begin{equation}
\label{choice-delta-n}
\delta=\delta_n \quad \text{such that} \quad \delta_n \to \infty  \  \text{  and  } \  \delta_n \frac1n \to 0 \ \text{ as } n \to\infty,
\end{equation}
but we will continue to  write $\delta$ in place of $\delta_n$ for shorter notation. \EEE

\RRR In order to show \eqref{lemma:thm-limsup-estimates}, we start from the very same algebraic manipulation as in 
\eqref{manipulation} and prove that the terms on the right-hand side converge to the desired limit. 
We observe that  \EEE
\[
\begin{split}
\bipname{\Psi_n}^\delta(\tau_n^\delta(s),\dot \su(s) , -  \rmD \E ( \st(s), \su(s)))   =  &\; \tau_n^\delta(s) \Psi_n \left( \frac{\dot \su(s)}{\tau_n^\delta(s)} \right) + \tau_n^\delta(s) \Psi_n^* (-\rmD \E(\st(s),\su(s))) + \tau_n^\delta(s) \delta \\
 \stackrel{\eqref{eq:explicit_tau_vv}}{=} & \psiri( \dot \su(s)) + \frac{\| \dot \su(s) \|}{2n} \sqrt{2n\delta + n^2 d_*(-\rmD \E(\st(s),\su(s)), K^*)^2} 
 \\ 
 &\quad
 + 
 \frac{n \| \dot \su(s) \| d_*(-\rmD \E(\st(s),\su(s)),K^* )^2}{2 \sqrt{2n\delta + n^2 d_*(-\rmD \E(\st(s),\su(s)), K^*)^2} } + \tau_n^\delta(s) \delta \\
\stackrel{n\to\infty}{\to} & \; \psiri(\dot \su(s)) + \| \dot \su(s) \| d_*(-\rmD \E(\st(s),\su(s)),K^*) \RRR 
\\
\leq 
&
\RRR \bpt{\dot{\st}(s)}{\dot{\su}(s)}{-\rmD \E(\st(s),\su(s))} 
\end{split}
\]
where the last inequality follows from the fact that   $\rmD \E(\st(s),\su(s)) \in K^*$ when $\dot \st(s) > 0$. Thus we conclude  the analogue of 
\eqref{1st-ingredient-manip}.  
Moreover, observe that, as a consequence of
the scaling for $\delta_n$ from
 \eqref{choice-delta-n}, we have 
\begin{equation}
\label{2nd-van-visc-manip}
\delta \sup_{s\in [0,S]} \tau_n^{\delta}(s) \to 0 \quad \text{as } n \to \infty. 
\end{equation}

\par
We then proceed to show the counterpart to \eqref{est_reg_case_1}.  The very same calculations as in 
\eqref{2n-ingredient}  (cf.\ also \eqref{ln-returned}),  \EEE give for every $s\in [0,S]$
\begin{equation}
\label{quoted-later-4}
\left| \rmD_i \E ( \st_n(s) \lambda_n^{-1}, \su(s))  -  \rmD_i \E  (\st(s),\su(s)) \right|
\leq C \sup_{s\in [0,S]} 
\tau_n^\delta(s)  \leq \frac{C}{\sqrt{n \delta}}\,.
\end{equation}
Resorting  now to  the explicit formula   \RRR  \eqref{explicit-psi_n}  for $\Psi_n^*$  (and using    $\xi_n(s) $
 and $\xi(s)$ as place-holders for 
   $ -\rmD \E ( \st_n(s) \lambda_n^{-1},\su(s))$
 and $ -\rmD \E ( \st(s), \su(s))$, respectively, to avoid overburdening notation),
 we get
\begin{equation}
\label{rompicapo}
\begin{aligned}
\dot \st_n(s) \Psi_n^*(\xi_n(s)) - \tau_n^\delta(s) \Psi_n^*( \xi(s))
& \stackrel{(1)}{=}
 \dot \st_n(s)   \frac n2   d_*( \xi_n(s),K^*)^2 
 \\
 &
  \stackrel{(2)}{\leq} 
    \dot \st_n(s) \frac n2   \| \xi_n(s) -\xi(s)  \|_*^2
    \\
    &
  \stackrel{(3)}{\leq} 
 \frac{C}{\delta} \qquad \foraa\, s \in  (\ss(0^+),\ss(T^-)). 
 \end{aligned}
\end{equation}
 In \eqref{rompicapo}, 
  (1) and (2) are due to the fact that $\rmD \E ( \st(s), \su(s)) \in K^*$  for almost all 
$ s \in  (\ss(0^+),\ss(T^-))$,  so that  $\Psi_n^*( \xi(s))=0$, 
and (3) to estimate \eqref{quoted-later-4}.  
To prove the analogue of  \eqref{est_reg_case_1},  we will first treat the case in which     $s \in [0,\ss(0^+)) \cup (\ss(T^-),S]$  (where 
$\dot{\st}_n(s) = \tau_n^\delta(s)$).
Here,  we use  the Lipschitz 
estimate \eqref{quoted-later-4} 
 and the explicit formula for $\Psi_n^*$. Thus,  we 
 find
\begin{equation}
\label{rompicapo-2}
\begin{split}
&
\tau_n^\delta(s) \left( \Psi_n^*(\xi_n(s)) - \Psi_n^*(\xi(s)) \right) 
\\
& = \frac{n}2 \tau_n^\delta(s) \left(
 d_*(\xi_n(s) , K^* )^2 - d_*( \xi(s), K^* )^2   \right) 
 \\
& 
 \leq  \frac{n}2 \tau_n^\delta(s) \left( \left( d_*(\xi_n(s) , \xi(s)) + d_* ( \xi(s), K^*) \right)^2
  - d_*( \xi(s), K^* )^2   \right) \\
   & \leq  \frac{n}2 \tau_n^\delta(s) \left( d_* ( \xi_n(s) , \xi(s) )^2 + 2d_* (\xi (s), K^*)  d_* ( \xi_n(s) , \xi (s))  \right) \\   
&  \stackrel{\eqref{quoted-later-4}}{\leq}   n \tau_n^\delta(s)  \left( \frac{C}{n\delta} + \frac{C}{\sqrt{n \delta}}\right)  \quad \text{for all }  s \in  [0,\ss(0^+)) \cup (\ss(T^-),S]
\end{split}
\end{equation}
 Combining \eqref{rompicapo} and \eqref{rompicapo-2}
 we infer \eqref{est_reg_case_1}, since $\delta=\delta_n \to \infty$  and $\frac{\delta_n}{n} \to 0 $ as $n\to\infty$.  
 In order to prove the analogue of \eqref{est_reg_case_2},  
we use the explicit formula \eqref{specific-1-limsup}
 of  $\Psi_n$ obtaining
\[
\begin{split}
\frac{\dot \st_n(s)}{\lambda_n}\Psi_n\left( \frac{\dot \su(s)}{\dot \st_n(s)}\lambda_n \right)
  & \stackrel{\dot \st_n \geq \tau_n^\delta}{\leq} \frac{\tau_n^\delta(s)}{\lambda_n}\Psi_n\left( \frac{\dot \su(s)}{\tau_n^\delta(s)}\lambda_n \right)  = \psiri(\dot \su(s)) + \frac{\lambda_n}{2n\tau_n^\delta(s)} \| \dot \su(s) \|^2 \\
& = \tau_n^\delta(s) \Psi_n \left( \frac{\dot \su(s)}{\tau_n^\delta(s)} \right) + (\lambda_n - 1) \frac{\| \dot \su(s) \|^2}{2n \tau_n^\delta(s)}.
\end{split}
\]
\RRR It follows from \eqref{eq:explicit_tau_vv}  and \eqref{choice-delta-n} that, for $n$ sufficiently big,   \EEE
\[
n \tau_n^\delta(s) \geq \frac{\| \dot \su(s) \|}{2 + d_*(-\rmD \E(\st(s),\su(s)), K^*)},
\]
hence we deduce that 
\begin{equation}
\label{last-tassello}
\frac{\dot \st_n(s)}{\lambda_n}\Psi_n\left( \frac{\dot \su(s)}{\dot \st_n(s)}\lambda_n \right) - \tau_n^\delta(s) \Psi_n \left( \frac{\dot \su(s)}{\tau_n^\delta(s)} \right) \leq  (\lambda_n - 1) C \qquad \foraa\, s \in (0,S).
\end{equation}
 \RRR Then,  \eqref{est_reg_case_2}, and ultimately  \eqref{lemma:thm-limsup-estimates}, ensue, since $\lambda_n \to 1$.  
 \par
 It remains to verify the integrated inequality \eqref{fatou-applied}.  Taking into account 
 \eqref{manipulation},  we observe that the 
 \emph{uniform} (w.r.t.\ $s\in (0,S)$) estimates \eqref{2nd-van-visc-manip},  \eqref{rompicapo},  \eqref{rompicapo-2},  and \eqref{last-tassello}, immediately allow for the application of the Fatou Lemma \eqref{Fatou}. Finally, we observe that, again by  the general representation formula 
 \eqref{supp-re}, there holds 
\[
\begin{split}
 \bipname{\Psi_n}(\tau_n^\delta(s),\dot \su (s), -  \rmD \E ( \st(s), \su(s))) & = \sup \left\{ \langle \xi, \dot \su (s) \rangle  \; | \; \xi \in K_{n,\delta}^*(-\rmD \E(\st(s),\su(s))) \right\}
\\
&
 \leq \sup \left\{ \| \dot \su(s)  \|_1 \| \xi \|\f \; | \; \xi \in K_{n,\delta}^*(-\rmD \E(\st(s),\su(s))) \right\} \\
& \leq C \| \dot \su(s) \|_1 \left( \| \rmD \E(\st(s),\su(s)) \|\f + \sqrt{ \frac{\delta}{n} } \right),
\end{split}
\]
and we conclude the integral estimate from the boundedness of $\| \dot \su(s) \|_1, \| \rmD \E (\st(s),\su(s))  \|\f$ and sending $n \to \infty$.
Thus, \eqref{fatou-applied} is proven.
\medskip

\noindent
\textit{Step $7$: recovery sequence for a general curve $u$ and conclusion of the proof.}
Now we  construct a recovery sequence for a curve with   countably many  jumps,  following the argument 
from 
the proof of \cite[Thm.\ 4.2]{BonaschiPeletier14}. 
 Given the jump set $\ju{u}$, fix $\e > 0$, consider a countable set   $\{ t^i \}_{i\in I}   \subseteq \ju u \cup \{0,T\}$ (with $t^i < t^{i+1}$ for all $i\in I$)  such that 
\begin{equation}
\label{eqn:condition_jump_e}
\mathrm{Jmp}_{\bptname,\calE}(u;[0,T] \setminus \{ t^i \} ) < \e,
\end{equation}
and such that the interval $[0,T]$ can be written as the union of disjoint subintervals
\[
[0,T]=\bigcup_{i\in I} \Sigma^i  \qquad \text{ where } \Sigma^i = [t^i,t^{i+1}].
\]
Then, we  let $t_n^i=\st_n(\ss(t^i))$, and  set
\[
\lambda_{n}^i : =\frac{t^{i+1}_n - t^i_n}{t^{i+1}-t^i}.
\]
We define the recovery sequence by
\begin{equation}
\label{eq:def_gen_rec_seq}
u_n(t):=\su \left( \st_{n}^{-1} \left( \lambda^i_n(t - t^i) + t^i_n \right) \right) \qquad \text{ for } t \in \Sigma^i,
\end{equation} 
so that
\begin{equation*}
\dot u_{n}(t)=\frac{\dot{\su}}{\dot{\st}_{n}}\left(  \st_{n}^{-1} \left(  \lambda^i_n(t - t^i) + t^i_n \right) \right) \lambda^i_{n} \qquad \text{ for } t \in (\Sigma^i)^\circ.
\end{equation*}
We have now that
\begin{align*}
\int_0^T \Bigl[ &\Psi_{n}\left( \dot u_{n}(t) \right) + \Psi_{n}^*\left({-} \rmD \E(t,u_{n}(t)) \right) \Bigr]\dd t \\
&= \sum_i \int_{\Sigma^i} \left[ \Psi_{n}\left( \frac{\dot{\su}}{\dot{\st}_{n}} \left( \st_{n}^{-1} \left( \lambda^i_n(t - t^i) + t^i_n \right)\right) \lambda^i_{n}  \right) + \Psi_{n}^*\left({-} \rmD \E\bigl(t,\su ( \st_{n}^{-1} (  \lambda^i_n(t - t^i) + t^i_n ))\bigr) \right) \right] \dd t \\
&=\sum_i \int_{\ss(t^i)}^{\ss(t^{i+1})}\left[ \Psi_{n} \left( \frac{\dot{\su}}{\dot{\st}_{n}}(s)\lambda^i_{n} \right) + \Psi_{n}^* \Bigl( {-}\rmD \E((\lambda^i_{n})^{-1}(\st_{n}(s)-t_{n}^i)+ t^i, \su(s)) \Bigr) \right]\frac{\dot{\st}_{n}(s)}{\lambda^i_{n}}\dd s.
\end{align*}
Applying  estimate \eqref{lemma:thm-limsup-estimates} in every subinterval  $[\ss(t^i),\ss(t^{i+1})]$ and Fatou's Lemma  \RRR (cf.\ 
\eqref{Fatou}) \EEE on
 the whole interval $[0,S]$,  we obtain inequality~\eqref{limsup-thm:ineq-to-prove}. 
\par
The convergence of the variations  again follows  by  the definition of $u_n$. \\
The pointwise convergence $u_n(t) \to u(t)$ for $t \in [0,T] \setminus \ju u $ is again trivial. The following calculations show that, by construction, the convergence holds also in the points $\{ t^i \} \subseteq \ju u $. Indeed, 
\[
u_{n}(t^i)\stackrel{\eqref{eq:def_gen_rec_seq}}{=} \su \left( \st_{n}^{-1} \left( t^i_n \right) \right) = \su \left( \st_{n}^{-1} \left( \st_n(\ss(t^i)) \right) \right) = \su \left( \ss(t^i) \right) \stackrel{\eqref{eq:inverserelation}}{=} u(t^i),
\]
while from \eqref{eqn:condition_jump_e} and the convergence of the variations we have that
\[
\lim_{n \to \infty} |u_n(t) - u(t) | < \e  \qquad \text{for all  }  t \in \ju u \setminus \{ t^i \}.
\]
In fact, the recovery sequence $u_n$ has a hidden dependence on $\e$. Then taking $\e = n^{-1}$ we define a new recovery sequence, that we keep labelling $u_n$, and sending  $n \to \infty$ ($\e$ to zero) we conclude. 
\medskip

\noindent This finishes the proof of Theorem \ref{thm:limsupgenerale}. 
\QED

\bibliographystyle{alpha}
\bibliography{multi_D_2014}

\end{document}